\documentclass{elsarticle}
  \usepackage{amssymb}
  \usepackage{amsthm}
  \usepackage{amsmath}

\newtheorem{theorem}{Theorem}[section]
\newtheorem{remark}{Remark}[section]

\newtheorem{lemma}{Lemma}[section]

\newcommand{\ul}[1]{\underline{#1}}
\newcommand{\ol}[1]{\overline{#1}}
\newfont{\smoldita}{cmmib8}
\newfont{\boldita}{cmmib10}
\newfont{\bboldita}{cmmib10 scaled\magstep1}

\newcommand{\mbb}[1]{\mathbb{#1}}
\newcommand{\mb}[1]{\mathbf{#1}}
\newcommand{\init}[1]{\stackrel{\circ}{#1}}
\newcommand{\nn}{\nonumber}
\newcommand{\e}{\epsilon}

\newcommand{\bom}{\mbox{\boldmath $\omega$}}

\newcommand{\ti}[1]{\tilde{#1}}

\newcommand{\ov}[1]{\overline{#1}}

\newcommand{\mc}[1]{\mathcal{#1}}

\newcommand{\comment}[1]{}
\journal{}
\begin{document}
\begin{frontmatter}
\title{Delayed stability switches in singularly perturbed predator-prey models}
\author[a,b]{J. Banasiak}  \author[c] {M. S. Seuneu Tchamga}
\address[a] {Department of Mathematics and Applied Mathematics, University of Pretoria, Pretoria, South Africa,  e-mail:
jacek.banasiak@up.ac.za} \address[b] {Institute of Mathematics,
Technical University of \L\'{o}d\'{z}, \L\'{o}d\'{z}, Poland}
 \address [c]  {School of Mathematical Sciences, University of
KwaZulu-Natal, Durban 4041, South Africa,  e-mail: 214539722@ukzn.ac.za}

\begin{abstract}
In this paper we provide an elementary proof of the existence of canard solutions for a class of singularly perturbed predator-prey planar systems in which there occurs a transcritical bifurcation of quasi steady states. The proof uses a one-dimensional theory of canard solutions developed by V. F. Butuzov, N. N. Nefedov and K. R. Schneider, and an appropriate monotonicity assumption on the vector field to extend it to the two-dimensional case. The result is applied to identify all possible predator-prey models with quadratic vector fields allowing for the existence of canard solutions. \end{abstract}

\begin{keyword}
 Singularly perturbed dynamical systems, multiple time scales, Tikhonov theorem,  delayed stability switch, non-isolated quasi steady states, predator-prey models, canard solutions.

\textit{2010 MSC}:     34E15, 34E17, 92D40.
\end{keyword}
\end{frontmatter}
\section{Introduction}
In many multiple scale problems; that is,  the problems in which processes occurring at vastly different rates coexist, the presence of such rates  is manifested by the presence of a small (or large) parameter which expresses the ratio of the intrinsic time units  of these processes. Mathematical modelling of such processes often leads to singularly perturbed systems of the form
 \begin{eqnarray}
{\mb x}' & = & \mb f(t,\mb x,\mb y,\e),\qquad \mb x(0)  = \mathring{ \mb x},\nn\\
\epsilon {\mb y}' & = & \mb g(t,\mb x,\mb y,\e),\qquad \mb y(0)  =  \,\mathring{\mb y},\label{(iss)}
\end{eqnarray}
where  $\mb f$ and $\mb g$ are sufficiently regular functions from open subsets of $\mbb R\times \mbb R^n\times \mbb R^m \times \mbb R_+$ to, respectively, $\mbb R^n$ and $\mbb R^m,$ for some $n,m\in \mbb N$. It is of interest to determine the behaviour of solutions to (\ref{(iss)}) as $\epsilon\to 0$ and, in particular, to show that they converge to solutions of the reduced system obtained from (\ref{(iss)}) by letting $\e=0$. There are several reasons for this. First, taking such a limit in some sense `incorporates' fast processes into the slow dynamics and hence links models acting at different time scales, often leading to new descriptions of nature, see e.g. \cite{BaLa14}. Second, letting formally $\e =0$ in (\ref{(iss)}) lowers the order of the system and hence reduces its computational complexity by offering an approximation that retains the main dynamical features of the original system. In other words, often the qualitative properties of the reduced system with $\e>0$ can be `lifted' to $\e>0$ to provide a good description of dynamics of (\ref{(iss)}).

The first systematic analysis of problems of the form (\ref{(iss)}) was presented by A.N. Tikhonov in the 40' and this theory, with corrections due to F. Hoppenstead, can be found in e.g. \cite{BaLa14, Ho, tivasv}. Later, a parallel theory based on the center manifold theory was given by F. Fenichel \cite{Fe} and a reconciliation of these two theories can be found in \cite{Sa}. To introduce the main topic of this paper one should understand the main features of either theory and, since our work is more related to the Tikhonov approach, we shall focus on presenting the basics of   it.

Let $\mb{\bar y}(t, \mb x)$ be the solution to the  equation
\begin{equation}
0 =  \mb g(t,\mb x,\mb y,0),
\label{deg1}
\end{equation}
often called the \textit{quasi steady state}, and $\mb{\bar x}(t)$ be the solution to
\begin{equation}
{\mb x}'  =  \mb f(t,\mb x,\mb{\bar y}(t, \mb x),0), \qquad \mb x(0)  =  \,\mathring{\mb x}.
\label{deg2}
\end{equation}
 We assume that $\bar{\mb  y}$ is an isolated solution to (\ref{deg1}) in some set $[0,T]\times \bar{\mc U}$  and that it is a uniformly, in  $(t,\mb x) \in [0,T]\times \bar{\mc U},$  asymptotically stable equilibrium of
 \begin{equation}\label{auxil'}
\frac{{d}\,\tilde{\mb y}}{{d}\,\tau} = \mb g(t,\mb x,\tilde{\mb y},0)\,,
\end{equation}
 where here $(t,\mb x)$ are treated as parameters.  Further,  assume that $\bar{\mb  x}(t) \in \mc U$ for $t\in [0,T]$ provided $\,\mathring{\mb x} \in \bar{\mc U}$ and that $\mathring {\mb y}$ is in the basin of attraction of $\mb{\bar y}$.
\begin{theorem}\label{tikhonovth'}
Let the above assumptions be satisfied. Then
there exists $\varepsilon_0 >0$ such that for any $\varepsilon\in\, ]\,0,\varepsilon_0 ]$ there
exists a unique solution
$(\mb x_{\varepsilon} (t),\mb y_{\varepsilon} (t))$ of (\ref{(iss)}) on $[0,T ]$ and
\begin{eqnarray}
\lim\limits_{\varepsilon\to 0} \mb x_{\varepsilon} (t) & =&\bar{\mb x}(t),\qquad t\in [0,T ]\,, \nn\\
\lim\limits_{\varepsilon\to  0} \mb y_{\varepsilon} (t) & =& \bar{\mb y}(t),
\qquad  t\in \, ]\, 0,T ]\,,\label{tikhonovcon'}
\end{eqnarray}
where $\bar{\mb x}(t)$ is the solution of  (\ref{deg2}) and $\bar{\mb y}(t)= \mb{\bar y}(t, \bar{\mb x}(t))$ is the solution of (\ref{deg1}).
\end{theorem}
We emphasize that the main condition for the validity of the Tikhonov theorem are that the quasi steady state be isolated and attractive; the latter in the language of dynamical systems is referred to as hyperbolicity. In applications, however, we often encounter the situation when either the quasi steady state ceases to be hyperbolic along some submanifold (a fold singularity), or two (or more) quasi steady states  intersect. The latter typically involves the so called `exchange of stabilities' as in the transcritical bifurcation theory: the branches of the quasi steady states change from being attractive to being repelling (or conversely)  across the intersection. The assumptions of the Tikhonov theorem fail to hold in the neighbourhood of the intersection, but it is natural to expect that any solution that passes close to it follows the attractive branches of the quasi steady states on either side of the intersection. Such a behaviour is, indeed, often observed, see e.g. \cite{but2004, LS1, LS2}. However, in many cases an unexpected behaviour of the solution is observed --- it follows the attracting part of one of the quasi steady states and, having passed the intersection, it continues along the now repelling branch of it for some prescribed time and only then jumps to the attracting part of the other quasi steady state. Such a behaviour, called the delayed switch of stability,  was first observed in \cite{Shi} (and explained in \cite{Nei}) in the case of a pitchfork bifurcation, in which an attracting quasi steady state produces two new attracting branches while itself continues as a repelling one. The delayed switch of stabilities for a fold singularity was observed in the van der Pol equation and have received  explanations based on methods  ranging from  nonstandard analysis \cite{BeCa} to classical asymptotic analysis \cite{E1}; solutions displaying such a behaviour were named \textit{canard solutions}. In this paper we shall focus on the so called transcritical bifurcation, in which two quasi steady states intersect and exchange stabilities at the intersection; here the delayed switch was possibly first observed in \cite{Hab} and analysed in \cite{Sch}.

The interest in problems of this type partly stems from the applications to determine slow-fast oscillations \cite{E1, Hek, Jo,  Mis, Rin1},  where the intersecting quasi steady states are used to prove the existence of cycles in the original problem and to approximate them. Another application is in the bifurcation theory, where the bifurcation parameter is driven by another, slowly varying, equation coupled to the original system \cite{BoTe1, BoTe2}. In both cases using the naive approximation of the true solutions by a solution lying on the quasi steady states, without taking into account the possibility of the delay in the stability switch, results in a serious under-, or overestimate of the real dynamics of the system, see e.g. \cite{BoTe1, BoTe2, Rin1}.

As we mentioned earlier, there is a rich literature concerning these topics and we do not claim that our paper offers significantly new theoretical results. However, by employing a monotonic structure of the equations and combining it with the method of upper and lower solution of \cite{but2004} we have managed to give a constructive and rather elementary proof of the existence of the delayed stability switch for a large class of planar systems including, in particular, predator-prey models with quadratic vector field.  As a by-product of the method, we also provided results on immediate stability switch. Here, our results pertain to a different class of problems than that considered in e.g. \cite{LS1, LS2} but, when applied to the predator-prey system, they give the same outcome. As an added benefit of our approach we mention that, in contrast to the papers based on the orbit analysis, e.g. \cite{Rin1}, we are able to give the precise value of time at which the stability switch occurs.  Finally we note that, for completeness, we only proved the results for planar systems. Some of them, however, can be extended to multidimensional systems, \cite{BaK2}.

The paper is structured as follows. In Section \ref{sec1} we recall the one-dimensional delayed stability switch theorem of \cite{but2004} and we formulate and prove its counterpart when the stability of quasi steady states is reversed. Section \ref{sec2} contains the main results of the paper. In Theorem \ref{th21} we prove the existence of the delayed switch for a general predator-prey type model. Theorem \ref{th3.2} shows the convergence of the solution to the second quasi steady state after the switch. Finally, in Theorem \ref{th33} we give conditions ensuring an immediate stability switch. In Section \ref{sec3} we apply these theorems to identify the cases of the delayed and immediate stability switches in classical predator-prey models. Finally, in Appendix we provide a sketch of the proof of the Butuzov \textit{et. al} result with some amendments necessary for our considerations.

\textbf{Acknowledgement.} The research of J.B. has been supported by the National Research Foundation CPRR13090934663. The results are part of Ph.D. research of M.S.S.T., supported by DAAD.

\section{Preliminary results}
\label{sec1}
\subsection{The one-dimensional result}
In this section we shall recall a  result on the delayed stability switch in a one dimensional case, given by V. F. Butuzov et.al., \cite{but2004}.
Let us consider a singularly perturbed scalar differential equation.
\begin{eqnarray}\label{eq1}
\epsilon \frac{dy}{dt} &=& g(t, y,\epsilon),\nn\\
y(t_0,\e)&=& \mathring y
\end{eqnarray}
in $D=I_{N}\times I_{T}\times I_{\epsilon_{0}},$ where $I_{N} =]-N,N[$,  $I_T = ]t_0, T[, I_{\epsilon_{0}}= \{\epsilon: 0< \epsilon< \epsilon_{0}<<1\},$ with $T>t_0, N>0$ and $g \in C^{2}(\bar{D},\mathbb{R}).$ Further, define
\begin{equation}
G(t,\epsilon)= \int _{t_{0}}^{t} g_{y}(s,0,\epsilon)ds.
\label{G}
\end{equation}
Then we adopt the following assumptions.
\begin{enumerate}
\item[$(\alpha_{1})$]  $g(t,y,0)=0$ has two roots  $y\equiv 0$ and $y=\phi(t)\in C^{2}(\bar I_{T})$ in ${I}_{N}\times \bar{I}_{T}$, which  intersect at $t=t_{c}\in (t_{0},T)$ and
$$
\phi(t)<0 \mbox{ for } t_{0}\leq t \leq t_{c},\quad \phi(t)>0 \mbox{ for } t_{c}\leq t \leq T.
$$

\item[$(\alpha_{2})$]
\begin{align*}
& g_{y}(t,0,0)<0, \; g_{y}(t,\phi(t),0)>0 \mbox{ for } t \in [t_{0},t_{c}),\\
&g_{y}(t,0,0)>0,\;  g_{y}(t,\phi(t),0)<0 \mbox{ for } t \in (t_{c}, T].
\end{align*}
\item[($\alpha_{3}$)]  $g(t,0,\epsilon)\equiv 0$ for $(t,\epsilon)\in \bar{I}_{T}\times \bar{I}_{\epsilon_{0}}.$
\item[($\alpha_{4}$)] The equation $G(t,0)=0$ has a root $t^{*} \in ]t_{0}, T[.$
\item[($\alpha_{5}$)] There is a positive number $c_{0}$ such that $\pm c_{0} \in I_{N}$ and
\begin{align*}
g(t,y,\epsilon)\leq g_{y}(t,0,\epsilon)y \mbox{ for } t \in [t_{0},t^{*}], \; \epsilon \in \bar{I}_{\epsilon_{0}}, |y|\leq c_{0}.
\end{align*}
\end{enumerate}

\begin{theorem}\label{thbtz}
Let us assume that all the assumptions ($\alpha_{1}$)--($\alpha_{5}$) hold. If $y_{0} \in (0,a),$ then for sufficiently small $\epsilon$ there exists a unique solution $y(t,\epsilon)$ of (\ref{eq1}) with
\begin{align}
\lim _{\epsilon \to 0} y(t,\e)=0 \mbox{ for } t \in ]t_{0},t^{*}[, \\
\lim _{\epsilon \to 0} y(t,\e)=\phi(t) \mbox{ for } t \in ]t^{*}, T],\label{btz2}
\end{align}
and the convergence is almost uniform on the respective intervals.
\end{theorem}
Some ideas of the proof of the above theorem play a key role in the considerations of this paper and thus we give a sketch of it in  \ref{secB}. Here we introduce essential notation and definitions which are necessary to formulate and prove the main results.

 By, respectively, lower and upper solutions to (\ref{eq1}) we understand continuous and piecewise differentiable (with respect to $t$) functions $\ul Y$ and $\ol Y$ that satisfy, for $t\in \bar I_T,$
\begin{eqnarray}
\ul Y(t,\e)\leq \ol Y(t,\e),&\quad&  \ul Y(t_0,\e)\leq \mathring y\leq  \ol Y(t_0,\e),\\
\epsilon \frac{d\ol Y}{dt} - g(t, \ol Y,\epsilon)\geq 0, &\quad& \epsilon \frac{d\ul Y}{dt} - g(t, \ul Y,\epsilon)\leq 0.
\end{eqnarray}
It follows that if there are upper $\ol Y$ and lower $\ul Y$ solutions to (\ref{eq1}), then there is a unique solution $y$  to (\ref{eq1}) satisfying
\begin{equation}
\ul Y(t,\e)\leq y(t,\e)\leq \ol Y(t,\e), \quad t\in \bar I_T, \e \in I_{\e_0}.
\label{exist}
\end{equation}
The proof of Theorem \ref{thbtz} uses an upper solution given by
\begin{equation}
\ov Y(t,\e) = \mathring u e^{\frac{G(t,\e)}{\e}}.
\label{upper}
\end{equation}
If we consider $\mathring y>0$ then, by assumption $(\alpha_3)$, $\ul Y=0$ is an obvious lower solution to (\ref{eq1}). It is, however, too crude to analyze the behaviour of the solution close to $t^*$ and the modification of (\ref{upper}), given by
\begin{equation}
\ul Y(t,\e) = \eta e^{\frac{G(t,\e)-\delta(t-t_0)}{\e}},
\label{ulY}
\end{equation}
is used, where $\eta, \delta$ are appropriately chosen constants.

 As explained in detail in \ref{secB}, conditions on $g$ can be substantially relaxed. Namely, we may assume that $g$ is a Lipschitz function on $\bar D$ with respect to all variables such that $g$ is twice continuously differentiable with respect to $y$ uniformly in $(t,y,\e)\in \bar D$ and that there is a neighbourhood of $(t^*,0)$, $V_{(t^*,0)}:= \;]t^*-\alpha, t^*+\alpha[\;\times\; ]-\e_1,\e_1[ $ in which $g_u(t,0,\e)$ is differentiable with respect to $\e$ uniformly in $t$.

\subsection{The case of reversed stabilities of quasi steady states}
It is interesting to observe that the phenomenon of delayed exchange of stability, described in Theorem \ref{thbtz}, does not occur if the role of the quasi steady states is reversed. Precisely, we have
\begin{theorem}\label{thbtz'}
Let us consider problem (\ref{eq1}) and assume
\item[$(\alpha_{1}')$]  $g(t,y,0)=0$ has two roots  $y\equiv 0$ and $y=\phi(t)\in C^{2}(\bar I_{T})$ in ${I}_{N}\times \bar{I}_{T}$, which  intersect at $t=t_{c}\in (t_{0},T)$ and
$$
\phi(t)>0 \mbox{ for } t_{0}\leq t \leq t_{c},\quad \phi(t)<0 \mbox{ for } t_{c}\leq t \leq T.
$$
Further, we assume that $(\alpha_2)$ and $(\alpha_3)$ are satisfied. Let $y_0\in (0,a)$. Then
\begin{eqnarray}
\lim _{\epsilon \to 0} y(t,\e)&=&\phi(t) \mbox{ for } t \in ]t_{0},t_c[, \nn\\
\lim _{\epsilon \to 0} y(t,\e)&=&0 \mbox{ for } t \in [t_c, T].\label{btz2'}
\end{eqnarray}
\end{theorem}
\begin{proof}
We see that $y=\phi(t)$ is an isolated attracting quasi steady state in the domain $[0,\bar t]\times [a_0, a],$ where $\bar t<t_c$
is an arbitrary number close to $t_c$ and $a_0>0$ is an arbitrary number that satisfies $a_0 < \inf_{t\in [0, \bar t]}\phi(t)$. Then $y_0>0$ is in the domain of attraction of $y=\phi(t)$. Hence, the first equation of (\ref{btz2'}) is satisfied. Let us take any $t'>t_c$. Then $y(t',\e)>0$ and thus it is in the domain of attraction of the quasi-steady state $y=0$. We cannot use directly the version of Tikhonov theorem, \cite[Theorem 1B]{LS1}, as we do not know a priori whether $y(t',\e)$ converges. In the one dimensional case, however, we can argue as in \ref{secB} to see that the second equation of (\ref{btz2'}) is satisfied on $]t_c, T]$. Finally, denoting by $\tilde \phi$ the composite attracting quasi steady state, $\tilde\phi(t) = \phi(t)$ for $t_0\leq t<t_c$ and $\tilde\phi(t) = 0$ for $t_c\leq t\leq T$, we see that $g(t,y,0)<0$ for $y>\ti \phi$ and thus, for $y>0$,   $g(t,y,\e)<0$ for $y >  \phi +\omega_\e$ with $\omega_\e \to 0$ as $\e \to 0$.
\end{proof}
\begin{remark} \emph{It is interesting to note that in this case the root $t^*$ of $G(t,0)$, see (\ref{G}), can satisfy $t^*>t_c$, but this does not have any impact on the switch of stabilities. Also, in general, the assumptions of theorem on an immediate switch of stabilities, e.g. \cite[Theorem 1.1]{but2004} are not satisfied,  see \cite[pp. 111--114]{EKP}.} \end{remark}
\section{Two-dimensional case}
\label{sec2}
\setcounter{equation}{0}
We consider the following singularly perturbed system of equations
\begin{eqnarray}
 x'(t) &=& f(t,x,y,\e), \nn\\
\e y'(t) &=& g(t,x,y,\e)\nn\\
x(t_0) &=& \mathring x, \quad y(t_0)= \mathring y.\label{sys1}
\end{eqnarray}
Let $V := I_T\times I_M\times I_N \times I_{\e_0} =]t_0,T[\;\times\; ]-M,M[\;\times\; ]-N,N[\;\times \;]0,\e_0[.$
 We introduce the following general assumptions concerning the structure of the system. Note that, apart the monotonicity assumptions (a3) and (a4), they are natural extensions of the assumptions of Theorem \ref{thbtz} to two dimensions.
 \begin{figure}[htp]
\begin{center}
  \includegraphics[scale = 0.35]{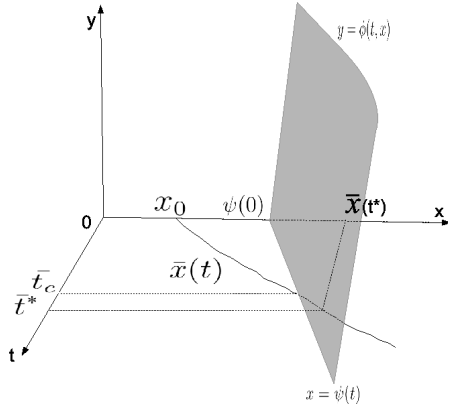}\\
  \caption{Illustration of the assumptions for Theorem \ref{th21}. }\label{ass1}
  \end{center}
\end{figure}

 \begin{description}
 \item (a1) Functions $f,g$ are $C^2(\ol V)$ for some $t_0<T \leq \infty, 0<M,N\leq \infty, \e_0>0$.
 \item (a2) $g(t,x,0,\e)=0$ for $(t,x,\e)\in I_T\times I_M\times I_{\e_0}.$
 \item (a3) $f(t,x,y_1,\e)\leq f(t,x,y_2,\e)$ for any $(t,x,y_1,\e), (t,x,y_2,\e) \in V, y_1\geq y_2$.
 \item (a4) $g(t,x_1,y,\e)\leq g(t,x_2,y,\e)$ for any $(t,x_1,y,\e), (t,x_2,y,\e) \in V, x_1\leq x_2$.
 \end{description}
 Further, we need assumptions related to the structure of quasi steady states of (\ref{sys1}).
 \begin{description}
 \item (a5) The set of solutions of the equation
 \begin{equation}
 0 = g(t, x,y,0)
 \label{degeq}
 \end{equation}
 in $\bar I_T\times \bar I_N\times \bar I_M$ consists of $y =0$ (see assumption (a2)) and $y = \phi(t,x),$ with $\phi \in C^2( \bar I_T\times\bar  I_M)$. The equation
 \begin{equation}
 0 = \phi(t,x)
 \label{switch}
 \end{equation}
 for each $t \in  \bar I_T$ has a unique simple solution $]0, M[ \ni x = \psi(t)\in C^2( \bar I_T)$. To fix attention, we assume that $\phi(t,x)<0$ for $x-\psi(t)<0$ and $\phi(t,x)>0$ for $x-\psi(t)>0$.
 \item (a6)
 $$
 \begin{array}{cccc}
 g_y(t,x,0,0)<0&\mathrm{and}& g_y(t,x,\phi(t,x),0)>0&\mathrm{for}\;x-\psi(t)<0,\\
 g_y(t,x,0,0)>0& \mathrm{and}&g_y(t,x,\phi(t,x),0)<0&\mathrm{for}\;x-\psi(t)>0.\end{array}$$
\end{description}
Since we are concerned with the behaviour of solutions close to the intersection of quasi steady state, we must assume that they actually pass close to it. Denote by $\bar x(t,\e)$ the solution of \begin{equation}
 x' = f(t, x, 0,\e), \quad  x(t_0,\e) = \mathring x.
\label{deg11}
\end{equation}
 Then we assume that
 \begin{description}
\item (a7) the solution $\bar x= \bar x(t)$ to the problem (\ref{deg11}) with $\e=0$, called the  reduced problem,
\begin{equation}
 x' = f(t,x,0,0), \quad x(t_0) =\, \mathring x
 \label{a7}
 \end{equation}
 with $-M<\,\mathring x\, <\psi(t_0)$ satisfies $\bar x(T)> \psi(T)$ and there is exactly one $\bar t_c \in ]t_0, T[$ such that $\bar x(\bar t_c) = \psi(\bar t_c)$.
\end{description}
Further, we define
\begin{equation}
\bar G(t,\epsilon)= \int _{t_{0}}^{t} g_{y}(s,\bar x(s,\e), 0,\epsilon)ds
\label{barG}
\end{equation}
and assume that
\begin{description}
\item (a8) the equation $$
\bar G(t,0)=\int _{t_{0}}^{t} g_{y}(s,\bar x(s), 0,0)ds=0$$ has a root $\bar t^{*} \in ]t_{0}, T[.$
\end{description}
As in the one dimensional case, by assumption (a6), $\bar G$ attains a unique negative minimum at $\bar t_c$ and is strictly increasing for $t>\bar t_c$ and thus assumption (a8) ensures that $\bar t^*$ is the only positive root in $]0,T[$.

Finally,
\begin{description}
\item (a9) There is  $0<c_{0}\in I_{N}$ and
\begin{align*}
g(t, \bar x(t,\e), y,\epsilon)\leq g_{y}(t, \bar x(t,\e), 0,\epsilon)y \mbox{ for } t \in [t_{0},\bar t^{*}], \; \epsilon \in \bar{I}_{\epsilon_{0}}, |y|\leq c_{0}.
\end{align*}
\end{description}
We noted earlier, though the list of assumptions is quite long, they are quite natural. Apart from usual regularity assumptions, assumptions (a5) and (a6) ensure that we have two quasi steady states with interchange of stabilities. Crucial for the proof are assumptions (a3) and (a4) that  allow to control solutions of (\ref{sys1}) by upper and lower solutions of appropriately constructed one dimensional problems, while (a7)-(a9) make sure that the latter satisfy the assumptions of Theorem \ref{thbtz}.

\begin{remark}\label{rem23}
\emph{In what follows repeatedly we will use the following argument which uses monotonicity of $f$ and $g$ in (\ref{sys1}) and is based on e.g.  \cite[Appendix C]{Walt}. Consider a system of differential equations
\begin{eqnarray}
x'=F(t,x,y),&\quad& x(t_0) =\mathring x,\nn\\
y'=G(t,x,y),&\quad& y(t_0) =\mathring y,
\label{sys6}
\end{eqnarray}
with $F$ and $G$ satisfying Lipschitz conditions with respect to $x,y$ in some domain of $\mbb R^2,$ uniformly in $t\in [t_0,T]$. Assume that $F$ satisfies $F(t,x,y_1)\leq F(t,x,y_2)$ for $y_1\geq y_2.$ If we know that a unique solution $(x(t),y(t))$ of (\ref{sys6}) satisfies $\phi_1(t, x(t))\leq y(t) \leq \phi_2(t,x(t))$ on $[t_0,T]$ for some Lipschitz functions $\phi_1$ and $\phi_2$, then $z_2(t) \leq x(t)\leq z_1(t),$ where $z_i$ satisfies
\begin{equation}
z_i' = F(t,z_i,\phi_i(t,z_i)), \qquad z_i(t_0) = \mathring x,
\label{zi}
\end{equation}
$i=1,2$.
Indeed, consider $z_1$ satisfying $z_1'(t)\equiv F(t,z_1(t),\phi(t,z_1(t))),\;\;z_1(t_0) = \mathring x.$ Then we have  $x'(t) \equiv F(t, x(t), y(t)) \leq F(t,x(t),\phi_1(t,x(t))$ and we can invoke \cite[Theorem B.1]{Walt} to claim that $x(t)\leq z_1(t)$ on $[t_0,T]$ (note that in the one dimensional case the so-called type $K$ assumption that is  to be satisfied by $F$ is always fulfilled). The other case follows similarly from the same result.\medskip\\
We also note that if $F$ satisfies $F(t,x,y_1)\leq F(t,x,y_2)$ for $y_1\leq y_2$ and we know that a unique solution $(x(t),y(t))$ of (\ref{sys6}) satisfies $\phi_1(t, x(t))\leq y(t)\leq \phi_2(t,x(t))$ on $[t_0,T]$ for some Lipschitz functions $\phi_1$ and $\phi_2$, then $z_1(t) \leq x(t)\leq z_2(t)$ where, as before, $z_i$ is a solution to (\ref{zi}).}
\end{remark}
\begin{theorem}\label{th21}Let assumptions (a1)-(a9) be satisfied and $-M<\init x <\psi(t_0), 0<\mathring y<N$.  Then  the solution $(x(t,\e),y(t,\e))$ of (\ref{sys1}) satisfies
\begin{eqnarray}
\label{eq4.82a}
  \lim_{\epsilon\rightarrow0}x(t,\e)&=&\bar{x}(t) \quad \mathrm{on}\;[t_0,\bar t^*[,\\
  \lim_{\epsilon\rightarrow0}{y}(t,\e)&=&0 \quad \mathrm{on}\;]t_0,\bar t^*[,\label{eq4.82}
  \end{eqnarray}
where $\bar{x}(t)$ satisfies (\ref{a7}) with $\bar{x}(t_0) = \mathring x$
and the convergence is almost uniform on respective intervals. Furthermore, $]t_0,\bar t^*[$ is the largest interval on which the convergence in (\ref{eq4.82}) is almost uniform.
\end{theorem}
\begin{proof}
   First we shall prove that there is $\bar t^*$ such that $y(t,\e) \to 0$ almost uniformly on $]0,\bar t^*[$. Let us fix initial conditions $(\mathring x, \mathring y)$ as in the assumptions and consider the solution $(x(t,\e), y(t,\e))$ originating from this initial condition.  Since $y(t,\e)\geq 0$ on $[t_0,T]$, assumption (a3) gives
\begin{equation}
x(t,\e) \leq \bar x(t,\e),
\label{xbound}
\end{equation}
 see (\ref{deg11}).
Then assumptions (a2) and (a4) give
\begin{equation}
0\leq y(t,\e) \leq \bar y(t,\e),
\label{yineq}
\end{equation}
where $\bar y(t,\e)$ is the solution to
\begin{equation}
\e  y' = \bar g(t, y, \e), \quad \bar y(t_0,\e) =\mathring y,
\label{eq1n}
\end{equation}
and we denoted $
\bar g(t,y,\e) := g(t,\bar x(t,\e), y, \e).$ Since (\ref{deg11}) is a regularly perturbed equation, by e.g. \cite{tivasv}, $(t,\e) \to \bar x(t,\e)$ is also twice differentiable with respect to both variables and thus $\bar g$ retains the regularity of $g$. Furthermore, $\bar g(t,y,0) = g(t,\bar x(t),y,0)$.

By (\ref{degeq}),  the only solutions to $\bar g(t,y,0) = 0$ are $y=0$ and $y= \phi(t, \bar x(t)).$ Denote $\varphi(t) = \phi(t, \bar x(t))$. From (\ref{switch}), $\phi(t,x) =0$ if and only if $x=\psi(t)$ and thus $\varphi(t)=0$ if and only if $\bar x(t)=\psi(t)$; that is, by (a7), for $t=\bar t_c$. Indeed, we have   $\varphi(\bar t_c)=\phi(\bar t_c, \bar x(\bar t_c)) = \phi(\bar t_c, \psi(\bar t_c)) =0,$ with $\varphi(t)<0$ for $t<\bar t_c$ and $\varphi(t)>0$ for $t>\bar t_c$. Hence, assumption $(\alpha_1)$ is satisfied for (\ref{eq1n}). Further, since $\bar g_y(t,y,\e) = g_y(t, \bar x(t,\e), y, \e)$, we see that assumption (a6) implies  $(\alpha_2)$. Then assumptions (a8) and (a9) show that assumptions $(\alpha_4)$ and $(\alpha_5)$ are satisfied for (\ref{eq1n}) and thus $\bar y(t,\e)$ satisfies (\ref{btz2}); in particular
$$
\lim _{\epsilon \to 0} \bar y(t,\e)=0 \mbox{ for } t \in ]t_{0},\bar t^{*}[. $$
This result, combined with (\ref{yineq}), shows that
$$
 \lim _{\epsilon \to 0}  y(t,\e)=0 \mbox{ for } t \in ]t_{0},\bar t^{*}[.
 $$
 Now,  for any $\mathring x$ satisfying (a7), there is a neighbourhood $U\ni \mathring x$ and  $\hat t  > t_0$ such that $y=0$ is an isolated quasi steady state  on $[t_0, \hat t]\times \bar U$ so that (\ref{sys1}) satisfies the assumptions of the Tikhonov theorem, see \cite{BaLa14}. Thus,
 $
 \lim_{\e\to 0} x(t,\e) = \bar x(t)
 $
 on $[t_0,\hat t]$ and hence the problem
 $$
 x' = f(t, x, y(t,\e),\e),
$$
 with initial condition $x(\hat t,\e)$ is regularly perturbed on $[\hat t, \bar t^*[$. Therefore, $\lim_{\e\to 0}x(t,\e) = \bar x(t)$ on $[\hat t, \bar t^*[$. Combining the above observations, we have
 $$
  \lim\limits_{\e\to 0} x(t,\e) = \bar x(t)
  $$
 almost uniformly on $[t_0,\bar t^*[$.

 In the next step we shall show that this is the largest interval on which $y(t,\e)$ converges to zero almost uniformly. Assume to the contrary that $\lim_{\e\to 0}y(t,\e) = 0$ almost uniformly on $]t_0, t_1]$ for some $t_1> \bar t^*$; that is,
 for any $\rho>0$ and any $\theta>0$ there is $\e_1 = \e_1(\rho,\theta)$ such that for any $t\in [t_0+\theta, t_1]$ and $\e<\e_1$ we have
 \begin{equation}
 0\leq y(t,\e) \leq \rho.
 \label{deltaest}
 \end{equation}
Then, by assumption (a3), on $[t_0+\theta, t_1]$ we have
$$
f(t, x, \rho, \e) \leq f(t, x, y(t,\e),\e).
$$
At the same time, $y(t,\e)\leq C$ for some constant $C>0$, see e.g. \cite[Proposition 3.4.1]{BaLa14}. In fact, in our case we see that $g<0$ for $y>0,$ sufficiently small $\e$  and $t$ close to $t_0$, hence $y(t,\e)\leq \mathring y$ on $[t_0,t_0+\theta]$ if $\theta$ is sufficiently small. Then the function \begin{equation}
\ul x^1(t) = \left\{\begin{array}{lcl} x_1(t,\e)&\mathrm{for}& t\in [t_0,t_0+\theta[,\\
 x_2(t,\e)&\mathrm{for}& t\in [t_0+\theta,t_1],
 \end{array}
 \right.
 \label{lowsol1}
 \end{equation}
 where $x_1' = f(t,x_1, C,\e), \; x_1(t_0) = \mathring x$ and $x_2' = f(t,x_2, \rho,\e), \; x_2(t_0+\theta) = x_1(t_0+\theta,\e)$ satisfies
 $\ul x^1 (t,\e)\leq x(t,\e)$. However, this function is not differentiable and cannot be used to construct the lower solution  for $y(t,\e)$.
Hence, we consider the solution $ \ul x_3$ to $\ul x_3{'} = f(t,\ul x_3, \rho,0), \; \ul x_3(t_0) = \mathring x$ on $[t_0,t_1]$. By Gronwall's lemma, using the regularity of $f$ with respect to all variables, we get
  \begin{equation}
 |\ul x^1(t,\e)-\ul x_3(t)| \leq L\theta
 \label{ulest}
 \end{equation}
 for some constant $L$ (note that $L$ can be made independent of $\e$ as $f$ is $C^2$ in all variables).  Thus, summarizing, for a given $\rho,$ there is $\theta_0$ such that for any $\theta< \theta_0$ and sufficiently small $\e$,
 \begin{equation}
-M<\ul x(t,\rho,\theta):=\ul x_3(t,\rho)-L\theta \leq \ul x^1(t,\e) \leq x(t,\e), \qquad t\in [t_0, t_1].
 \label{ulx}
 \end{equation}
Then, using assumption (a4), we find that the solution $\ul y = \ul y(t,\rho,\theta,\e)$ to
\begin{equation}
\e \ul y'= \ul g(t, y, \rho,\theta,\e), \quad \ul y(0,\rho,\theta,\e) = \mathring y,
\label{lowsoly}
\end{equation}
 where $\ul g(t, y, \rho,\theta,\e) := g(t,\ul x(t,\rho,\theta), y,\e),$ satisfies
 $$
 \ul y(t,\rho,\theta,\e) \leq y(t,\e), \qquad t \in [t_0, t_1].
 $$
By construction, equation (\ref{lowsoly}) is in the form allowing for the application of Theorem \ref{thbtz}. We will not need, however, the full theorem but only the considerations for the lower solution.
As with $\bar g$, we note that $\ul g$ is a $C^2$ function with respect to all variables. We consider the function
\begin{equation}
\ul G(t,\rho,\theta,\epsilon)= \int _{t_{0}}^{t} \ul g_{y}(s, 0,\rho, \theta,\epsilon)ds
\label{ulG}
\end{equation}
and observe that $\ul g(t,0,0,0,\e) = \bar g(t,\e)= g(t, \bar x(t,\e),0,\e)$ and also $\ul g_{y}(t, 0,0,0,\epsilon) = \bar g_y(t,\e)= g_y(t, \bar x(t,\e),0,\e).$ Then $\ul G(t_0,\rho,\theta,0) =0$. Further, since $\ul G(\bar t^*,0,0,0) = \bar G(\bar t^*,0) =0$ and $\ul G_t(\bar t^*,0,0,0) = g_y(\bar t^*,0,0) >0$, the Implicit Function Theorem shows that for sufficiently small $\rho,\theta$ there is a $C^2$ function $\ul t^*=\ul t^*(\rho,\theta)$ such that $\ul G(\ul t^*(\rho,\theta),\rho,\theta,0) \equiv 0$ with $\ul t^*(\rho,\theta) \to \bar t^*$ as $\rho,\theta \to 0$.

Furthermore, since by (a4) and (a2) we have $g(t,x_1,y,0)\leq g(t,x_2,y,0)$ for $x_1\leq x_2$ and $g(t,x,0,0)=0$, we easily obtain
\begin{equation}
g_y(t,x_1,0,0) \leq g_y(t,x_2,0,0), \quad x_1\leq x_2.
\label{niergy}
\end{equation}
Since
$$
\ul x (t,\rho,\theta) \leq x(t,\e) \leq \bar x(t), \quad t \in [t_0, t_1]
$$
we find that $\ul G(t,\rho,\theta,0) \leq \bar G(t,0)$ and thus $\ul t^*(\rho,\theta)\geq \bar t^*$.

Denote by $\ul Y(t,\rho,\theta, \delta, \eta,\e)$ the solution defined by (\ref{ulY}) with $G$ replaced by $\ul G.$  We observe that the parameter $\delta$ is defined independently of $\rho, \theta$ and $\eta,$ hence $\ul G(t(\rho,\theta,\delta,\e),\rho,\theta,\e) - \delta(t-t_0) \equiv 0$ and
$$
\ul Y(t(\rho,\theta,\delta, \e),\rho,\theta, \delta, \eta,\e) = \eta.$$
This function $\ul Y$ is a lower solution to (\ref{lowsoly}) provided $\eta\leq \delta/k$, see (\ref{k}), where $k$ can be also made independent of any of the parameters.  So, we can find $\rho_0,\theta_0$ such that
$$
\sup\limits_{0\leq \rho\leq \rho_0, 0\leq \theta\leq \theta_0} \ul t^*(\rho,\theta) \leq \ti t <t_1.
$$
Then, for a given $\rho,\theta$ satisfying the above, we have
$$
t(\rho,\theta,\delta, \e) = t^*(\rho,\theta) + \omega(\delta,\e)
$$
and we can take $\delta, \e_1$ such that $\omega(\delta,\e)+ \ti t<t_1$ for all $\e<\e_1$. For such a $\delta$, we fix $\eta<\delta/k$ and then $\rho<\eta$. Then, for sufficiently small $\e$, $y(t(\rho,\theta,\delta, \e), \e)<\rho$ and, on the other hand,
$$y(t(\rho,\theta,\delta, \e), \e)\geq \ul Y(t(\rho,\theta,\delta, \e),\rho,\theta, \delta, \eta,\e) = \eta >\rho.$$
Thus, the assumption that there is $t_1>\bar t^*$ such that $y(t,\e)$ converges almost uniformly to zero on $]t_0,t_1[$ is false.\end{proof}

In the next step, we will investigate the behaviour of the solution beyond $\bar t^*$. Clearly, we cannot use  $\ul y$ defined by (\ref{lowsoly}) as a lower solution there since it is a lower solution only as long as $x(t,\e)\leq \rho$ which, as we know, is only ensured for $t<\bar t^*$. Thus, we have to find another \textit{a priori} upper bound for $x(t,\e)$ that takes into account the behaviour of $x(t,\e)$ beyond $\bar t^*$. For this we need to adopt an additional assumption which ensures that $x(t,\e)$ does not return to the region of attraction of $y = 0$. Let
\begin{equation}
\left.\frac{g_t}{g_x}+f\right|_{(t,x,y,\e)= (t, \psi(t), 0,0)}>0, \quad t\in [0,T].
\label{a9}
\end{equation}
\begin{remark} \label{remcross} \emph{Condition (\ref{a9}) has a clear geometric interpretation, see Fig.\ref{ass1}.  The normal to the curve $x = \psi(t)$ pointing towards the region $\{(t,x);\;x>\psi(t)\}$ is given by $(-\psi'(t), 1)$. However, we have $0 \equiv \phi(t,\psi(t))$, hence $\psi'=\left.-\phi_t/\phi_x\right|_{(t,x)=(t,\psi(t))}$ which, in turn, is given by $-g_t/g_x$ on $(t,x,y,\e) = (t,\psi(t), \phi(t,\psi(t)), 0) = (t,\psi(t), 0, 0)$ on account of $0
\equiv g(t,x,\phi(t,x), 0)$. Thus (\ref{a9}) is equivalent to
$$
(-\psi',1)\cdot (1,x') = (-\psi',1)\cdot (1,f), \quad (t,x,y,\e)= (t, \psi(t), 0,0),
$$
so that it expresses the fact that the solution $x$ of (\ref{a7}) cannot cross $x = \psi(t)$ from above. If the problem is autonomous, then (\ref{a9}) turns into
$$
\left. f\right|_{(x,y,\e)= (c, 0,0)}>0, \quad t\in [0,T],
$$
where  $x=\psi(\bar t_c)\equiv c$, which means that $\bar x(t)$ is strictly increasing crossing the line $x = c.$ }
\end{remark}
\begin{theorem}
Assume that, in addition to (a1)--(a9), inequality (\ref{a9}) is satisfied. Then
\begin{eqnarray}
\label{eq4.82b}
  \lim_{\epsilon\rightarrow0}{x}(t,\e)={x}_\phi(t), \quad ]\bar t^*, T],\\
  \lim_{\epsilon\rightarrow0}{y}(t,\e)=\phi(t,{x}_\phi(t)),\quad ]\bar t^*, T],\label{eq4.82c}
  \end{eqnarray}
where $\bar{x}_\phi(t)$ satisfies
\begin{equation}
 x_\phi' = f(t, x_\phi, \phi(t,x_\phi),0), \quad  x_\phi(\bar t^*)= \bar x(\bar t^*)
\label{xphi}
\end{equation}
and the convergence is almost uniform on $]\bar t^*,T]$.
\label{th3.2}
\end{theorem}
\begin{proof} Since the proof is quite long, we shall begin with its brief description. Note that in the notation here we suppress the dependance of the construction on all auxiliary parameters.  The idea is to use the one dimensional argument, as in Theorem \ref{th21}; that is,  to construct an appropriate lower solution but this time on $[t_0,T]$. As mentioned above, for $t<\bar t^*$ we can use $\ul x$ and $\ul y,$ but beyond $\bar t^*$ we must provide a new construction. First, using the classical Tikhonov approach, we show that if $y(t,\e),$ with sufficiently small $\e,$ enters the layer  $\phi-\omega<y<\phi+\omega$ at some $t>\bar t_c$, then it stays there. Hence, in particular, we obtain an upper bound for $y(t,\e)$ for $t>t_c$. Combining it with the upper bound obtained in the proof of Theorem \ref{th21},  we obtain an upper bound for $y$ on $[t_0,T]$ which is, however, discontinuous. Using (a3), this gives a lower solution $\ul X$ for $x(t,\e)$ on $[t_0, T],$ that can be modified to be a differentiable function. It is possible to prove that $\ul X$ stays uniformly bounded away from $\psi$ but only up to some $\ti t>\bar t^*$. This fact is essential as otherwise the equation for $\ul Y,$ constructed using $\ul X$ as in (\ref{lowsoly}), would have quasi steady states intersecting in more than one point (whenever $\ul X(t) =\psi(t)$, see the considerations following (\ref{eq1n})). Hence, we only can continue considerations on $[t_0,\ti t\,]$. Now, as in the one dimensional case, the constructed $\ul Y$ converges on $]t_0,\ti t\,]$ to some quasi steady state, which is close to $\phi(t,\ul X(t))$ but, since we only have $y(t,\e)\geq \ul Y(t,\e)$, this is not sufficient for the convergence of $y(t,\e)$. However, this estimate allows for constructing an upper solution for $x(t,\e)$ and hence an upper solution for $y(t,\e)$. By careful application of the regular perturbation theory for $x(t,\e)$ we prove that $y(t,\e)$ is sandwiched between two functions which are small perturbations of $\phi(t,x_\phi(t))$, where $x_\phi$ satisfies  (\ref{xphi}). Thus $y(t,\e)$ converges to $\phi(t,x_\phi(t))$ on $]t_0,\ti t\,]$. This shows, in particular, that the solution enters the layer $\phi-\delta<y<\phi+\delta$ for arbitrarily small $\delta$ provided $\e$ is small enough, and the application of the Tikhonov approach with a Lyapunov function allows for extending the convergence up to $T$.

 \textbf{Step 1. An upper bound for $y(t,\e)$ after $\bar t_c$. }Let us take arbitrary $t_1 \in ]\bar t_c, \bar t^*[$. By (\ref{a9}), there is $\varrho_0>0$ such that $\bar x(t_1)> \psi(t_1) +\varrho_0$.  Since  $x(t_1,\e) \to \bar x(t_1)$ and $y(t_1,\e)\to 0,$ there is $\e_0$ such that for any $0<\e<\e_0$ we have $x(t_1,\e)>\psi(t_1) +\rho_0/2$ and $0<y(t_1,\e)<\rho,$ as established in the proof of the previous theorem.
Let
$$
\Psi(t,x,y,\e) := \frac{g_t(t,x,y,\e)}{g_x(t,x,y,\e)}+f(t,x,y,\e).
$$
\begin{figure}[htp]
\begin{center}
  \includegraphics[scale = 0.25]{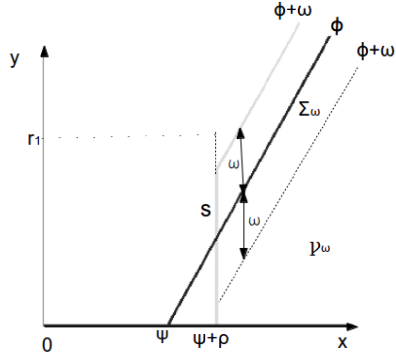}\\
  \caption{The cross-section of the construction for a given $t$. }\label{ass2}
  \end{center}
\end{figure}
By (\ref{a9}), we have $\Psi(t, \psi(t), 0,0)>0$ for $t\in [0,T]$ and thus there is $\alpha_1$, $r_1,r_2,\e_0$ such that
\begin{equation}
\Psi(t,\psi(t)+\varrho,y,\e)\geq \alpha_1
\label{Psi}
\end{equation}
for all $|y|\leq r_1, |\varrho|<r_2$, $|\e|<\e_0$. Consider now the surface $S = \{(t,x, y);\; t\in [0,T], x = \psi(t)+\varrho, 0\leq y\leq r_1\}$. By continuity, there is $0<\varrho <\min\{\varrho,r_2\}$ such that
$$
\max\limits_{t\in [0,T]}\phi(t, \psi(t)+\varrho) <r_1.
$$
Let
$$
\alpha_\varrho= \min\limits_{t\in [0,T], \psi(t)+\varrho\leq x\leq M}\phi(t, x)>0
$$
and, for arbitrary $0<\omega <\min\{\alpha_\varrho/2, r_1-\max\limits_{t\in [0,T]}\phi(t, \psi(t)+\varrho),$ consider the layer
\begin{equation}
\Sigma_\omega = \{(t,x,y);\;t\in [0,T], \psi(t) +\varrho \leq x\leq M, \phi(t,x)-\omega \leq y \leq \phi(t,x)+\omega\}.
\label{siom}
\end{equation}
and the domain
$$
\mc V_\omega =  \{(t,x,y);\;t\in [0,T], \psi(t) +\varrho \leq x\leq M, 0 \leq y \leq \phi(t,x)+\omega\}.
$$
Note that `left' wall of $\mc V_\omega$, $\mc V_{\omega,l}:=\mc V_\omega\cap S$ is contained in the set $\{(t,x,y);\Psi(t,x,y,\e)>0\}$ and thus, by Remark \ref{remcross}, no trajectory can leave $\mc V_\omega$ across $\mc V_{\omega,l}$. Using a standard argument with the Lyapunov type function $V(t) = (y(t,\e)-\phi(t, x(t,\e)))^2$, see e.g. \cite[pp. 86-90]{BaLa14} or \cite[p. 203]{tivasv}, if the solution is in $\Sigma_\omega$, it cannot leave this domain through the surfaces $y = \phi(x,t)\pm \omega.$  Hence, in particular, we have $\{x(t,\e), y(t,\e)\}_{t_1 \leq t\leq T } \in \mc V_\omega.$

 \textbf{Step 2. Construction of the lower solution for $x(t,\e)$ on $[t_0,T]$.} By Step 1, for an arbitrary fixed $t_1\in ]\bar t_c,\bar t^*[,$ there is $\omega$ such that  $y(t,\e)$;  $0<y(t,\e)< \phi(t, x(t,\e))+\omega$ for $t\in [t_1, T].$  On the other hand, for any $\rho>0$ and sufficiently small $\theta>0$ we have $0<y(t,\e)<\rho$ on $[t_0+\theta, \bar t^*-\theta]$ for all $\e<\e_1 = \e(\rho,\theta)$. Then, by (\ref{ulx}),  we have in particular $\ul x(t,\theta, \rho) \leq x(t,\e)$ for $t\in [t_0,\bar t^*-\theta].$

Consider now the solution to
$$
x_4' = f(t,x_4, \phi(t,x_4)+\omega,\e), \quad x_4(\hat t) = \ul x(\hat t,\theta, \rho), \quad t\in [\hat t,T],
$$
for some some $\hat t\in ] t_1, \bar t^*-\theta[$. Using Remark \ref{rem23}, we see that  $x_4(t,\theta, \rho, \e) \leq x(t,\e)$ for all sufficiently small  $\e$. At the same time, using the regular perturbation theory, for any $\vartheta>0$ there is,  possibly smaller, $\e_5$ such that for all $\e<\e_5$ and $t\in [\hat t, T]$ the solution $x_5(t) = x_5(t,\hat t, \theta, \rho)$ to
\begin{equation}
x_5' = f(t,x_5,\phi(t,x_5),0), \quad x_5(\hat t) = \ul x(\hat t,\theta, \rho), \quad t\in [\hat t,T],
\label{x5}
\end{equation}
satisfies
$$
|x_5(t,\theta,\rho) - x_4(t,\theta,\rho,\e)|< C\vartheta
$$
on $[\hat t, T]$, with $C$ independent of $\theta, \rho, \e, \vartheta, \hat t$. Then we construct the function
$$
\ul X(t,\theta,\rho,\vartheta) = -C\vartheta + \left\{\begin{array}{lcl} \ul x(t,\theta, \rho)&\mathrm{for}& t\in [t_0,\hat t],\\x_5(t,\theta, \rho)&\mathrm{for}& t\in ]\hat t,T],\end{array}\right.
$$
which clearly satisfies
\begin{equation}
\ul X(t,\theta,\rho,\vartheta)  \leq x(t,\e), \quad t\in [t_0, T].
\label{lowsolx}
\end{equation}
Next we prove that $\ul X$ stays uniformly away from $\psi(t)$ in some neighbourhood of $\bar t^*$. For this, we note that both $\bar x$ and $\ul x$ are defined on $[t_0, T]$ and close to each other, by the definition of $x_3$ and (\ref{ulx}) (for small $\rho$). Thus, by (a7), there are $\Omega''\leq \Omega'$ and $t^\#<\bar t^*$ such that $\bar x\geq \psi + \Omega'$ and $\ul x\geq \psi + \Omega''$ on $[t^\#,\bar t^*].$ Let $0<\Omega<\Omega''$. Then, by (a1), we see that $\inf_{\bar V}f\geq K$ for some $K>-\infty$ (which follows, in particular, since  $0\leq y(t,\e)\leq \phi(t,x(t,\e))$ for $t\geq \hat t$) and hence
$$
x_5(t) \geq x_5(\hat t) + K(t-\hat t).
$$
Then, for any $\hat t \in ]t^\#,\bar t^*[$, we have
\begin{eqnarray*}
\ul X(t,\theta,\rho,\vartheta) &=& x_5(t) -C\vartheta  \geq x_5(\hat t) + K(t-\hat t) = \ul x(\hat t) + K(t-\hat t)\geq \psi(\hat t) + \Omega'' + K(t-\hat t)
\\&=&\psi(t)+\Omega +(\psi(\hat t)-\psi(t) + K(t-\hat t) - C\vartheta + \Omega'' -\Omega).
\end{eqnarray*}
Since the constants $C, \Omega, \Omega''$ can be made independent of $\hat t\in [t^\#, \bar t^*],$ and by the regularity of $\psi,$ we see that there is $\ti t> \bar t^*,$  $\hat t$ sufficiently close to $\bar t^*,$ and $\vartheta>0$ such that
\begin{equation}
\ul X(t,,\theta,\rho,\vartheta)\geq \psi(t)  +\Omega, \quad t\in [\hat t, \ti t\,].
\label{posx5}
\end{equation}
\textbf{Step 3. Construction of the lower solution for $y(t,\e)$ on $[t_0,T]$ and its behaviour for $t\in ]\bar t^*, \ti t\,]$.} Let us now consider the solution $\ul Y(t,\theta,\rho,\vartheta,\e)$ of the Cauchy problem
\begin{equation}
\e\ul Y' = g(t, \ul X(t,\theta, \rho, \vartheta), \ul Y,\e), \qquad \ul Y(t_0,\theta,\rho,\vartheta, \e) = \mathring y.
\label{globulY}
\end{equation}
We observe that the above equation has two  quasi-steady states, $y\equiv 0$ and $y = \phi(t,\ul X(t,\theta, \rho, \vartheta)),$ that only intersect at $\ul t_c,$ which is close to
$\bar t_c, $ at least on $[t_0,\ti t\,]$. Moreover, for $t<\hat t$ the lower solution $\ul x$ can be made as close as one wishes to $\bar x$.
Though $\ul X$ is not a $C^2$ function, as required by  Theorem \ref{thbtz}, we can use the comment at the end of \ref{secB} and only consider $t\geq \hat t$. Here, instead of only a Lipschitz function $\ul X$, we have the function $x_5(t,\theta,\rho)-C\vartheta$ that is smooth with respect to all parameters -- note that $\rho$ and $\theta$ enter into the formula through a regular perturbation of the equation and the initial condition.  We define the function $\ul {\mc G}$ for (\ref{globulY}) by
\begin{equation}
\ul {\mc G}(t,\rho,\theta,\vartheta,\epsilon)= \int _{t_{0}}^{t} g_y (s, \ul X(s,\theta, \rho, \vartheta), 0,\e)ds.
\label{globulG}
\end{equation}
We observe that for $t<\hat t$ we have, by (\ref{niergy}),
$$
\ul {\mc G}(t,\rho,\theta,\vartheta,0) = \int _{t_{0}}^{t} g_y (s, \ul x(s,\theta, \rho)-C\vartheta, 0,0)ds\leq \ul G(t,\rho, \theta, 0).
$$
and also, since $\ul X(t) \leq x(t,\e)\leq \bar x(t)$ for any $t\in [t_0,T]$,
\begin{equation}
\ul {\mc G}(t,\rho,\theta,\vartheta,0) \leq \bar G(t,0).
\label{ulgbarg}
\end{equation}
This means that $\ul{\mc G}<0$ on $]0,\hat t]$ and $\ul{\mc G} \to 0$ with $\hat t \to \bar t^*$ and $\theta, \rho, \vartheta \to 0$. Now, writing
$$
\ul {\mc G}(t,\rho,\theta,\vartheta,0) = \int _{t_{0}}^{\hat t} g_y (s, \ul x(s,\theta, \rho)-C\vartheta, 0,0)ds + \int _{\hat t}^{t} g_y (s, x_5(t,\theta, \rho)-C\vartheta, 0,0)ds
$$
and, using (a6) and (\ref{posx5}) to the effect that $g_y (t, x_5(t,\theta, \rho)-C\vartheta, 0,0) \geq L$ on $[\hat t, \ti t\,]$ for some $L>0,$ we see that for sufficiently small $\bar t^*-\hat t, \theta, \rho$ and $\vartheta$ we have
$$
\int_{\bar t^*}^{\ti t} g_y (s, x_5(s,\theta, \rho)-C\vartheta, 0,0)ds \geq L(\ti t-\bar t^*) > \int _{t_{0}}^{\hat t} g_y (s, \ul x(s,\theta, \rho)-C\vartheta, 0,0)ds,
$$
since the last term is negative. Therefore there is a solution $\ul t^* = \ul t^*(\hat t, \rho, \theta, \vartheta)<\ti t$ to $\ul {\mc G}(t,\rho,\theta,\vartheta,0) =0$. Moreover, this solution is unique as $\ul {\mc G}$ is strictly monotonic for $t\geq \hat t$, by (\ref{ulgbarg}) it satisfies  $\ul t^*>\bar t^*$ and $\ul t^* \to \bar t^*$ if  $\bar t^*-\hat t, \theta, \rho, \vartheta \to 0$. Now, for a fixed $\hat t, \rho, \theta, \vartheta$, $\ul G$ is a $C^2$-function of $(t,\e)\in ]\hat t, \bar t^*[\times\; ]\!-\bar \e, \bar \e[$ where $\bar \e$ is chosen so that (\ref{lowsolx}) is satisfied for all $0<\e<\bar \e$. Thus, we can apply Theorem \ref{thbtz} with the weaker assumptions discussed at the end of \ref{secB}  to claim that
\begin{equation}
\lim\limits_{\e \to 0} \ul Y(t,\theta,\rho,\vartheta, \e) = \phi(t, x_5(t)-C\vartheta)
\label{ulYconv}
\end{equation}
almost uniformly on $]\ul t^*, \ti t\,]$. Because of this, for any $\tau\in ]\ul t^*, \ti t[$ and any $\delta'>0$ we can find $\ti \e>0, \ti \vartheta >0$ such that for any $\e<\ti \e, \vartheta <\ti \vartheta$ and $t\in [\tau, \ti t]$ we have
\begin{equation}
y(t,\e)\geq  \phi(t, x_5(t)) -\delta'.
\label{1lest}
\end{equation}
 \textbf{Step 4. Upper solutions for $x(t,\e)$ and $y(t,\e)$ on $[t_0,\ti t\,]$.} Thanks to these estimates, we see that the solution $x_6 = x_6(t,\e)$ of the problem
\begin{equation}
x_6' = f(t,x_6, \phi(t,x_5)-\delta',\e), \quad x_6(\tau,\e) = \bar x(\tau, \e)
\label{x6}
\end{equation}
satisfies, for sufficiently small $\e$,
$$
x_6(t,\e) \geq x(t,\e)
$$
on $t\in [\tau,\ti t\,]$. Thus, we can construct a composite upper bound for $x(t,\e)$ on $[t_0,\ti t\,]$ as
$$
\bar X(t,\e) = \left\{\begin{array}{lcl} \bar x(t,\e)&\mathrm{for}& t\in [t_0,\tau],\\x_6(t,\e)&\mathrm{for}& t\in ]\tau,\ti t\,]
\end{array}\right.
$$
and hence a new upper bound for $y(t,\e),$ defined to be the solution to
\begin{equation}
\e\bar Y' = g(t, \bar X(t,\e), \bar Y,\e), \qquad \bar Y(t_0, \e) = \mathring y.
\label{globulY1}
\end{equation}
We observe that for $t\in [t_0,\tau]$ we have
$$
g(t, \bar X(t), 0,0) = g(t,\bar x(t),0,0).
$$
Hence
\begin{equation}
\bar {\mc G}(t,0)= \int _{t_{0}}^{t} g_y (s, \bar X(s,0), 0,0)ds
\label{globarG0}
\end{equation}
coincides with $\bar G(t,0)$ on $[t_0,\tau]$ with $\tau>\bar t^*$ and thus $\bar{\mc G}(\bar t,0) <0$ for $t\in ]t_0, \bar t^*[$, $\bar{\mc G}(\bar t^*,0) =0$ and $\bar{\mc G}(\bar t,0) >0$ for $t\in ]\bar t^*,\ti t[$ since, by (\ref{lowsolx}) and (\ref{posx5}),  $x(t,\e)>\psi(t)$ on $[\bar t^*,\tau]$ and $x_6(t,\e)> \psi(t)$ on $[\tau, \ti t\,]$. Thus the assumptions of Theorem \ref{thbtz} are satisfied and we see that
\begin{equation}
\lim\limits_{\e \to 0} \bar Y(t,\e) = \phi(t, x_6(t,0))
\label{barYconv}
\end{equation}
uniformly on $[\tau, \ti t\,]$.

\textbf{Step 5. Convergence of $(x(t,\e), y(t,\e))$ on $]\bar t^*, \ti t\,]$.} Now, $x_6(t,0)$ is the solution to
\begin{equation}
x_6' = f(t,x_6, \phi(t,x_5)-\delta',0), \quad x_6(\tau,\e) = \bar x(\tau, 0),
\label{x60}
\end{equation}
which is a regular perturbation of
$$
x' = f(t,x,\phi(t,x_5),0), \quad x(\hat t) = \ul x(\hat t,\theta, \rho), \quad t\in [\hat t,T].
$$
But, by the uniqueness, the solution of the latter is $x_5$ and thus, for any $\delta''>0$ we can find $\hat t,\tau, \theta, \rho, \vartheta, \delta',\e''$ such that for all $\e<\e''$ we have
$$
|x_6(t,0)- x_5(t)|<\delta''
$$
on $[\tau, \ti t\,].$ We need some reference solution independent of the auxiliary parameters so we denote by $x_\phi$ the function satisfying
$$
x_\phi' = f(t,x_\phi, \phi(x_\phi),0), \quad  x_\phi(\bar t^*)= \bar x(\bar t^*)
$$
Clearly, this equation is a regular perturbation of both (\ref{x60}) and (\ref{x5})
and thus for any $\delta'''>0$, after possibly further adjusting $\e$, we find
\begin{equation}
\phi(t,x_\phi(t))-\delta''' \leq y(t,\e) \leq \phi (t,x_\phi(t)) +\delta''', \quad t\in [\tau, \ti t\,]
\label{phi1}
\end{equation}
which shows that
\begin{equation}
\lim\limits_{\e \to 0}  y(t,\e) = \phi(t, x_\phi(t))
\label{yconv}
\end{equation}
uniformly on $t\in [\tau, \ti t\,]$. This in turn shows that
\begin{equation}
\lim\limits_{\e \to 0} x(t,\e) =  x_\phi(t)
\label{yconv'}
\end{equation}
uniformly on $t\in [\tau, \ti t\,].$

\textbf{Step 6. Convergence of $(x(t,\e), y(t,\e))$ on $]\bar t^*, T]$.}  Eq. (\ref{yconv'}) allows us to re-write (\ref{phi1}) as
$$
\phi(t,x(t,\e))-\ti\delta \leq y(t,\e) \leq \phi (t,x(t,\e)) +\ti \delta, \quad t\in [\tau, \ti t\,],
$$
for some, arbitrarily small, $\ti \delta>0$. Using the argument with the Lyapunov function and the notation from Step 1, the trajectory will not leave the layer $\Sigma_{\ti\delta}$. But then, by the standard argument as in e.g. \cite[pp. 86-90]{BaLa14}, we obtain
\begin{equation}
\lim\limits_{\e \to 0} x(t,\e) =  x_\phi(t)
\label{xconv}
\end{equation}
uniformly on $t\in [\tau, T]$ and consequently
$$
\lim\limits_{\e \to 0}  y(t,\e) = \phi(t, x_\phi(t))
$$
uniformly on $t\in [\tau, T]$. Since we could take $\tau >\bar t^*$ arbitrarily close to $\bar t^*$, we obtain the thesis.
\end{proof}

Next, we provide a two-dimensional counterpart of Theorem \ref{thbtz'}, in which the stability of the quasi steady states is reversed. It provides conditions for an immediate switch of stabilities but, due to the structure of the problem, covers a different class of problems than e.g. \cite[Theorem 2]{LS1} or \cite[Theorem 1.1]{but2004}. More precisely, we have   \begin{theorem}\label{th33}
Consider problem (\ref{sys1}) with assumptions (a1), (a2), (a8)-(a9), (\ref{a9}) and
\begin{description}
 \item (a5') The solution of the equation
 \begin{equation}
 0 = g(t, x,y,0)
 \label{degeq'}
 \end{equation}
 in $\bar I_T\times \bar I_N\times \bar I_M$ consists of $y =0$ and $y = \phi(t,x),$ where $\phi \in C^2( \bar I_T\times\bar  I_M)$. The equation
 \begin{equation}
 0 = \phi(t,x)
 \label{switch'}
 \end{equation}
 for each $t \in  \bar I_T$ has a unique simple solution $]0, M[ \ni x = \psi(t)\in C^2( \bar I_T)$. We assume that $\phi(t,x)>0$ for $x-\psi(t)<0$ and $\phi(t,x)<0$ for $x-\psi(t)>0$.
 \item (a6')
 $$
 \begin{array}{cccc}
 g_y(t,x,0,0)>0&\mathrm{and}& g_y(t,x,\phi(t,x),0)<0&\mathrm{for}\;x-\psi(t)<0,\\
 g_y(t,x,0,0)<0& \mathrm{and}&g_y(t,x,\phi(t,x),0)>0&\mathrm{for}\;x-\psi(t)>0.\end{array}$$
 \item (a7') The solution $x_\phi$ to the problem
\begin{equation}
 x' = f(t,x,\phi(t,x),0), \quad x(t_0) =\, \mathring x,
 \label{a7'}
 \end{equation}
 with $-M<\,\mathring x\, <\psi(t_0)$ satisfies $x_\phi(T)> \psi(T)$ and there is exactly one $t_c \in ]t_0, T[$ such that $x_\phi(t_c) = \psi(t_c)$.
\end{description}
Then  the solution $(x(t,\e),y(t,\e))$ of (\ref{sys1}) satisfies\\
(a)  \begin{eqnarray}
  \lim_{\epsilon\rightarrow0}{x}(t,\e)&=&{x}_\phi(t) \quad \mathrm{on\;}[t_0, t_c[,\nn\\
  \lim_{\epsilon\rightarrow0}{y}(t,\e)&=&\phi(t,{x}_\phi(t))\quad \mathrm{on\;}]t_0,t_c[,\label{eq4.82d}
  \end{eqnarray}
and the convergence is almost uniform on respective intervals;\\
(b) \begin{eqnarray}
  \lim_{\epsilon\rightarrow0}x(t,\e)&=&\bar{x}(t) \quad \mathrm{on\;}[t_c,T], \nn\\
  \lim_{\epsilon\rightarrow0}{y}(t,\e)&=&0 \quad \mathrm{on\;}[t_c,T[,\label{eq4.82e}
  \end{eqnarray}
where $\bar{x}(t)$ satisfies (\ref{a7}) with $\bar{x}(t_c) = x_\phi(t_c)$
and the convergence is uniform.\label{mth2}
\end{theorem}
\begin{proof}
Some technical steps of the proof are analogous to those in the proofs of Theorems \ref{th21} and \ref{th3.2} and thus here we shall give only a sketch of them.

From (a7') we see that for any $\ul t_c <t_c$ there is $\delta_{\ul t_c}$ such that $\inf_{t_0\leq t\leq \ul t_c}(\psi(t)-x_\phi(t))\geq \delta_{\ul t_c}.$  For any $0<\eta<\delta_{\ul t_c}$ define $U_\eta =\{(t,x);\; t_0\leq t\leq \ul t_c, 0\leq x\leq \psi(t)-\eta\}$. By (a5'), we have
$$
\xi_\eta = \inf\limits_{(t,x)\in U_\eta}\phi(t,x)>0
$$
and thus $\phi$ is an isolated quasi steady state on $U_\eta.$ Note that in the original formulation of the Tikhonov theorem, \cite{BaLa14, tivasv}, $U_\eta$ should be a cartesian product of $t$ and $x$ intervals, but the current situation can be easily reduced to that by the change of variables $z(t)=x(t)-\psi(t)$. Hence, (\ref{eq4.82d}) follows from the Tikhonov theorem. We observe that for any $\eta>0$ we can find $\ul t_c$ so that $y(\ul t_c,\e)<\eta$ and $\psi(\ul t_c)-\eta< x(\ul t_c,\e)<\psi(\ul t_c)+\eta$.

Now, as in (\ref{Psi}), there are $\alpha_1>0, \zeta_0, \e_0$ such that \begin{equation}
\Psi(t,\psi(t)+\zeta,y,\e)\geq \alpha_1
\label{alpha}
\end{equation}
for all $|y|\leq \zeta_0, |\zeta|<\zeta_0$, $|\e|<\e_0$.

Further, denote by $\ti \phi$ the composite stable quasi steady state: $\tilde\phi(t,x) = \phi(t,x)$ for $t_0\leq t<T, 0<x\leq \psi(t)$ and $\tilde\phi(t,x) = 0$ for $t_0\leq t\leq T, \psi(t)<x\leq M$.  Then, by (a6'), we see that $g(t,x, y,0)<0$ for $t_0\leq t\leq T, 0\leq x\leq M, \ti \phi(t,x)<y\leq N.$ Therefore, for any $\omega>0$ there is $\beta>0$ such  $g(t,x, y,0)<-\beta$ for $y\geq \ti\phi +\omega$ and thus also  $g(t,x,y,\e)\leq 0$ for $y \geq  \ti\phi +\omega$ for sufficiently small $\e$.

Now, let us take arbitrary $\zeta<\zeta_0,$ $\omega<\zeta$ and $\eta$ such that $\phi(t,\psi(t)-\eta)+\omega<\zeta$. Then we take  $\ul t_c$ such that $x(\ul t_c, \e) > \psi(\ul t_c)-\eta.$ It is clear that $y(t,\e)\leq \zeta$ for $t\geq \ul t_c$. Indeed, by (\ref{alpha}), the trajectory cannot cross back through $\{(t,x,y);\; t_0\leq t\leq T,  x  =\psi(t)-\eta, 0\leq y\leq \phi(t,\psi(t)-\eta)+\omega\}$, hence the only possibility would be to go through $\ti\phi + \omega< \eta$ for $x> \psi(t)-\eta$  but then, by the selection of constants, the trajectory would enter the region where $y'(t,\e) \leq 0$.  Thus, a standard argument shows that
$$
\lim\limits_{\e\to 0} y(t,\e) = 0,
$$
uniformly on $[\ul t_c, T]$. Then the problem
$$
x' = f(t,x,y(t,\e),\e), \quad x(\ul t_c,\e) = x(\ul t_c,\e)
$$
on $[\ul t_c, T]$ is a regular perturbation of
$$
x' = f(t,x,0,0), \quad x(t_c) = x_\phi(t_c),
$$
whose solution is $\bar x.$ Therefore (\ref{eq4.82e}) is satisfied.

Using (\ref{alpha}) we can get a more detailed picture of the solution. Indeed, we see that
$$
x(t,\e) > \psi(t)+\eta
$$
for $t< \bar t_c:=\ul t_c + 2\eta/\alpha_1$ and sufficiently small $\e$ and $(x(t,\e),y(t,\e))$ cannot cross back through
 $\{(t,x,y);\; t_0\leq t\leq T,  x  =\psi(t)+\eta, y\geq 0\},$ by $0\leq y(t,\e)\leq \zeta$ for $t\geq \ul t_c.$ Thus the solution stays in the  domain of attraction of the quasi steady state $y=0$ after $x(t,\e)$ crosses the line $x =\psi(t)$.

\end{proof}
\section{An application to predator--prey models}
\label{sec3}
\setcounter{equation}{0}
Let us consider a general mass action law model of two species interactions,
\begin{eqnarray}
x'&=& x(A+Bx+Cy),\quad x(0)= \mathring x,\nn\\
\e y'&=& y(D+Ey+Fx), \quad y(0) = \mathring y,
\label{gen1}
\end{eqnarray}
where none of the coefficients equals zero. It is natural to consider this system in the first quadrant $Q=\{(x,y);\;x\geq 0,y\geq 0\}$. It is clear that  $y=0$ is one quasi steady state, while the other is given by the formula
$$
y = \phi(t,x)=-\frac{F}{E}x -\frac{D}{E},
$$
with $\psi(t) = -D/F$.
This quasi steady state lies in $Q$ only if $-D/F>0$. Under this assumption, the geometry of Theorem \ref{th21} is realized if $-F/E>0,$ while that of Theorem \ref{mth2} if $-F/E<0$. At the same time,
$$
g_y(x,y) = D+2Ey+Fx.
$$
Hence, $g_y(x,0)<0$ if and only if $D+Fx<0$, while $g_y(x,\phi(t,x)) <0$ if and only if $D+Fx>0$.

Summarizing, for the switch to occur in the biologically relevant region, $D$ and $F$ must be of opposite sign. In what follows we use positive parameters $a,b,c,d,e,f$ do denote absolute values of capital case ones.  Then we have the following cases.\smallskip\\
\textbf{Case 1.} $D<0, F>0$.\\
\textbf{Case 1a.} $E>0$. Then the right hand side of the second equation in (\ref{gen1}) is of the form $y(-d+ey+fx)$ with $y$ describing a predatory type population but with a very specific vital dynamics. It may describe a population of sexually reproducing  generalist predator, see e.g. \cite[Section 1.5, Exercise 12]{Braun}, but  its dynamics is not very interesting -- without the prey it either dies out or suffers a blow up. Also, in the coupled case of (\ref{gen1}), the only attractive quasi steady state in $Q$ is $y=0$ for $x <d/f$ as the attracting part of $\phi$ is negative. We shall not study this case.\\
\textbf{Case 1b.} $E<0$. In this case the right hand side of the second equation of (\ref{gen1}) is of the form $y(-d-ey+fx)$ which may describe a specialist predator (one that dies out in the absence of a particular prey). In this case the second quasi steady state is given by
$$
y = \phi(x)=\frac{f}{e}x -\frac{d}{e},
$$
and the quasi steady state $y=0$ is attractive for $x<d/e$ and repelling for $x>d/e,$ where $\phi$ becomes attractive. Hence we are in the geometric setting of Theorem \ref{th21}. For its applicability, $f(x,y) = x(A+Bx+Cy)$ must be decreasing with respect to $y$, which requires $C<0$ (for $x>0$). Then the assumptions of Theorem \ref{th21} require either $A,B>0,$ or $A>0,B<0$ with $a/b>d/f$
with $0<\mathring x <d/f$, or $A<0,B>0$ with $a/b<d/f$ and $a/b<\mathring x< d/f$, as in each case the solution $\bar x$ to
\begin{equation}
x'= x(A+Bx), \quad x(0) = \mathring x
\label{bareq}
\end{equation}
crosses $d/f$ at some finite time $t_c$. We observe that $\bar x$ is increasing in all three cases. Thus, the function $G$, defined by (\ref{G}), satisfies
$$
G''(t) = f\bar x'(t) >0
$$
and thus there is a  unique $t^*>t_c$ for which $G(t^*)=0.$ Finally, we see that $g_{yy}(x,y) = -e<0$ and thus (a9) is satisfied.
\begin{figure}[htp]
\begin{center}
  \includegraphics[scale = 0.5]{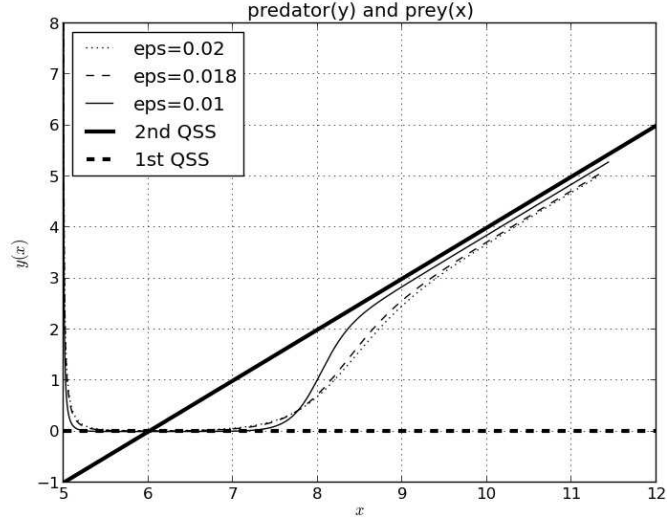}\\
  \caption{Delayed stability switch in the \textbf{Case 1b}. The orbits are traversed from left to right. }\label{incr}
  \end{center}
\end{figure}

\textbf{Case 2}. $D>0$ and $F<0$.\\
\textbf{Case 2a.} $E>0$. Then the right hand side of the second equation in (\ref{gen1}) is of the form $y(d+ey-fx),$ thus  $y$ describes a prey type population but with a specific vital dynamics: if not preyed upon, $y$ blows  up in finite time. Also, in the coupled case of (\ref{gen1}), the only attractive quasi steady state in $Q$ is $y=0$ for $x >d/f$ as the attracting part of $\phi$ for $x <d/f$ is negative. As before, we shall not study this case.\\
\textbf{Case 2b.} $E<0$. Here, the right hand side of the second equation of (\ref{gen1}) is  $y(d-ey-fx),$ which describes a prey with logistic vital dynamics. The second quasi steady state is given by
\begin{equation}
y = \phi(x)=-\frac{f}{e}x +\frac{d}{e},
\label{phi2a}
\end{equation}
and the quasi steady state $y=0$ is repelling for $x<d/e$ and attractive for $x>d/e,$  while $\phi$ is attractive for $x<d/e$. Thus the geometry of the problem is that of Theorem \ref{mth2} and we have to identify conditions on $A,B$ and $C$ that ensure that the solution $x_\phi$, see (\ref{a7'}),  originating from $\mathring x<d/f,$ crosses the line $x=d/f$ in finite time. In this case (\ref{a7'}) is given by
\begin{equation}
x' = x\left(\frac{Ae+Cd}{e} + \frac{Be-Cf}{e} x\right).
\label{bu}
\end{equation}
Consider the dynamics of this equation. If $Be-Cf=0,$ then there is only one equilibrium $x=0$ and the solution grows or decays depending on whether $Ae+Cd$ is positive or negative. If $Be-Cf\neq 0$, then there is another equilibrium, given by
 $$
 x_{eq}=-\frac{Ae+Cd}{Be-Cf}.
 $$
 The assumptions of Theorem \ref{mth2} will be satisfied if and only if $\mathring x<d/f$ and $x_{eq}>d/f$ and it is attracting, or $Be-Cf=0, Ae+Cd>0$, or  $x_{eq}\in [0,d/f[$ is repelling with $x_{eq}<\mathring x$.

To express these conditions in algebraic terms, we see that if $Be-Cf\neq 0$, then we must have
\begin{equation}
-A< B\frac{d}{f},
\label{AB}
\end{equation}
 while if $Be-Cf = 0$, then $B$ and $C$ must be of the same sign and for the solution to be increasing we must have $Ae+Cd>0$ which again  yields (\ref{AB}). Summarizing,  (\ref{AB}) is equivalent to $B=b>0, A=a>0$, or $B=b>0,A =-a<0$ {and} ${a}/{b}<d/{f},$ {or} $B=-b<0, A=a>0$ {and} ${a}/{b}>{d}/{f}.$
It is important to note that these conditions do not involve the position of $x_{eq}$. Just to recall, we must have either  $x_{eq}>d/f$ and it is attracting, or $x_{eq}<d/f$ and it is repelling (here we can think of the case $Be-Cf=0$ with $A,C>0$ as having $x_{eq}=-\infty$.)
Thus, assumptions of Theorem \ref{mth2} are satisfied  if  and only if the geometry is as in this point, (\ref{AB}) is satisfied and $\mathring x \in \,]x_{eq}, d/f[\,$ if $x_{eq}<d/f$. Then the $x$ component of the solution $(x(t,\e),y(t,\e))$ to (\ref{gen1})  grows above $d/f$ and an immediate change of stability occurs when the solution passes close to $(d/f,0)$.

We note that \textbf{Case 2b} can be transformed to a problem that satisfies the assumptions of \cite[Theorem 2]{LS1}. On the other hand, not all assumptions of \cite[Theorem 1.1]{but2004} are satisfied.

It is interesting that \textbf{Cases 1b} and \textbf{2b} have, in some sense, their duals. Consider, in the geometry of \textbf{Case 2 b}, $\mathring x>d/f$ and assume that the coefficients are such that the solution $\bar x(t)$ to (\ref{bareq}) decreases and crosses $d/f$. Then the solutions $(x_\e(t),y_\e(t))$ are first attracted by $(\bar x(t),0)$ as long as they are above $x>d/f$ and later they enter the region of attraction of (\ref{phi2a}). So, under some technical assumptions, one can expect again a delay in the exchange of stabilities. We prove this by transforming this case to \textbf{Case 1b}. Hence, consider (\ref{gen1}) in the geometric configuration of \textbf{Case 2b},
\begin{eqnarray}
x'&=& x(A+Bx+Cy),\quad x(0)= \mathring x\nn\\
\e y'&=& y(d-ey-fx), \quad y(0) = \mathring y,
\label{gen2b}
\end{eqnarray}
and assume that $\mathring{x}>0.$ Then the solution $\bar x$ to
$$
x'= x(A+Bx), \quad x(0) =\mathring x,
$$
will decrease and pass through $x =d/f$ if and only if $-A>Bd/f$ (which is equivalent to either $A=-a<0,B=-b<0$, or $A=a>0, B=-b <0$ and $a/b<d/f$, or $A=-a<0, B=b>0$ and $a/b>d/f$) and $\mathring x<a/b$ in the latter case.
Let us change the variable according to $x= -z+{2d}/{f}.$
Then the system (\ref{gen2b}) becomes
\begin{eqnarray}
z'&=& \left(z-\frac{2d}{f}\right)\left(\frac{Af+2Bd}{f}-Bz+Cy\right),\quad z(0)= \frac{2d}{f}-\mathring x<\frac{d}{f},\nn\\
\e y'&=& y(-d-ey+fz), \quad y(0) = \mathring y.
\label{gen2}
\end{eqnarray}
We observe that the second equation is the same as in \textbf{Case 1b}, so the assumptions of Theorem \ref{th21} concerning the function $g$ are satisfied. We only have to ascertain that the assumptions concerning the function $f$ of Theorem \ref{th21} also hold. We note that we consider the problem for $z<2d/f$ where the multiplier $(z-2d/f)<0$. Thus, to have (a3) we need $C=c>0$. For (a7), we observe that the equlibria of $z$ are $z_1 = 2d/f$ and
$$
z_2 = \frac{A}{B} + \frac{2d}{f}.
$$
As before, (a7) will be satisfied if $z_2<d/f$ is repelling with $\mathring z > z_2$, or $z_2>d/f$ and is attracting, or $z_2> 2d/f$ and $z_1$ is attracting. It is easy to see that the first case occurs when $A/B<-d/f$ and $B>0$, the second when $A/B>-d/f$ and $B<0$, and the last when both $A>0,B>0$. Thus, we obtain
$$
-A>B\frac{d}{f}.
$$
Since  the case when $z_2<d/f$ and it is repelling is possible if and only if  $B=b>$ and $A=-a<0$, we see  that $d/f>\mathring z  > z_2$ is equivalent to $d/f<\mathring x < a/b$.

We observe that if we consider the geometry of \textbf{Case 1 b}, but assume that $\mathring x >d/f$ and the solution to
\begin{equation}
x' = x\left(\frac{Ae-Cd}{e} + \frac{Be+Cf}{e} x\right), \quad x(0) =\mathring x
\label{bu1}
\end{equation}
is decreasing and passes through $x=d/f$, then, by the same change of variables as above, we can transform this problem to the one discussed in \textbf{Case 2b} and obtain that there is an immediate switch of stabilities as in Theorem \ref{mth2}.
\begin{figure}[htp]
\begin{center}
  \includegraphics[scale = 0.5]{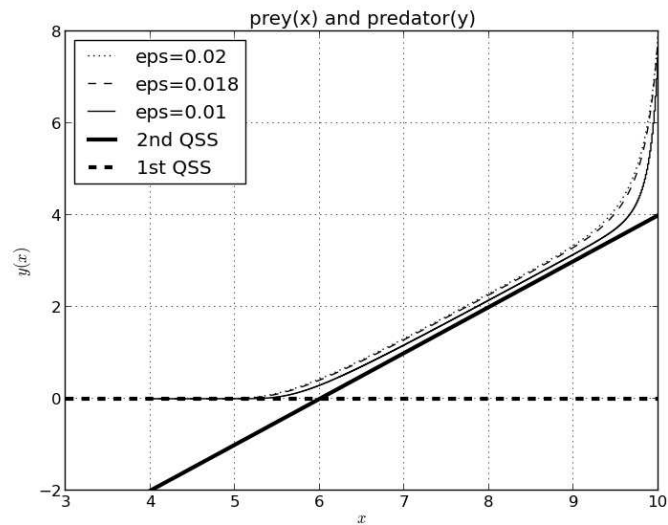}\\
  \caption{Stability switch without delay in the geometry of the case \textbf{Case 1b} with $\mathring x>d/f$. The orbits are traversed from right to left.  }\label{decr}
  \end{center}
\end{figure}

To summarize, we obtain the delayed switch of stabilities in the following six cases:\smallskip\\
{\textbf{Fast predator}}\\
\textbf{a)}
\begin{eqnarray*}
x'&=& x(a+bx-cy),\quad x(0)= \mathring x\in ]0,d/f[,\nn\\
\e y'&=& y(-d-ey+fx), \quad y(0) = \mathring y>0,
\end{eqnarray*}
\textbf{b)}
\begin{eqnarray*}
x'&=& x(a-bx-cy),\quad x(0)= \mathring x\in ]0,d/f[,\nn\\
\e y'&=& y(-d-ey+fx), \quad y(0) = \mathring y>0,
\end{eqnarray*}
with $a/b>d/f$,\\
\textbf{c)}
\begin{eqnarray*}
x'&=& x(-a+bx-cy),\quad x(0)= \mathring x\in ]a/b,d/f[,\nn\\
\e y'&=& y(-d-ey+fx), \quad y(0) = \mathring y>0,
\end{eqnarray*}
with $a/b<d/f$.\smallskip\\
{\textbf{Fast prey}}\\
\textbf{a)}
\begin{eqnarray*}
x'&=& x(-a-bx+cy),\quad x(0)= \mathring x>d/f,\nn\\
\e y'&=& y(d-ey-fx), \quad y(0) = \mathring y>0,
\end{eqnarray*}
\textbf{b)}
\begin{eqnarray*}
x'&=& x(a-bx+cy),\quad x(0)= \mathring x>d/f,\nn\\
\e y'&=& y(d-ey-fx), \quad y(0) = \mathring y>0,
\end{eqnarray*}
with $a/b<d/f$,\\
\textbf{c)}
\begin{eqnarray*}
x'&=& x(-a+bx+cy),\quad x(0)= \mathring x\in ]d/f,a/b[\nn\\
\e y'&=& y(d-ey-fx), \quad y(0) = \mathring y>0,
\end{eqnarray*}
with $a/b>d/f$.

\appendix
\section{}
\label{secB}
\textbf{Sketch of the proof of Theorem \ref{thbtz}.}
 To explain the construction of  the upper solution (\ref{upper}), first we observe that, by the Tikhonov theorem, for any $c_0>0$ (see assumption $(\alpha_5)$) and $\delta >0$ (such that $t_0+\delta<t_c$), there is an $\e(\delta)>$ such that $0<y(t_0+\delta,\e) \leq c_0$. Thus, using ($\alpha_3$), all solutions $y(t,\e)$ are nonnegative and bounded from above by the solution of (\ref{eq1}) with $\bar t =t_0+\delta$ and $\mathring v_b = c_0$. Since in the first identity of (\ref{btz2}) we have to prove the convergence on the open interval $]t_0,T],$ it is enough to prove it for any $\delta$ with the initial condition at $t_0+\delta$ being smaller than $c_0$. Thus, without loosing generality, we can assume that $y(t_0,\e)  =\mathring y \leq c_0$. Then assumption $(\alpha_5)$ asserts that the right hand side of (\ref{eq1}) is dominated by its linearization at $y=0$ as long as the solution remains small (that is, at least on $[t_0, \hat t]$ for any $\hat t<t_c$). The author then considers the linearization
$$
\epsilon \frac{d\ov Y}{dt} = g_y(t, 0,\epsilon)\ov Y,\quad \ov Y(t_0,\e)= \mathring u \in ]0,c_0],
$$
whose solution is (\ref{upper}), $\ov Y(t,\e) = \mathring u \exp \e^{-1}{G(t,\e)}.$
Crucial for the estimates are the properties of $G$. From the regularity of $g$ and $(\alpha_2)$ we see that $g_y(t,0,\e)$ is negative and separated from zero for sufficiently small $\e$ and thus, by (\ref{G}), $G(t,\e)\leq 0$ on $[t_0,t_0+\nu]$ for some small $\nu >0$. Similarly, from $(\alpha_4)$ and the regularity of $G$ with respect to $\e$
we find that there is a constant $\kappa$ such that
\begin{equation}
\frac{G(t,\e)}{\e} \leq \frac{G(t,0)}{\e} +\kappa
\label{Glip}
\end{equation}
on $[t_0, t^*]$, $\e \in I_{\e_0},$ so that
$G(t,\e)/\e<0$ on $[t_0+\nu, t^*-\nu]$ for sufficiently small $\e$. Hence $\ov Y(t,\e) \leq c_0$ on $[t_0,t^*-\nu]$ and sufficiently small $\e,$ and thus the inequality of assumption $(\alpha_5)$ can be extended on $[t_0,t^*-\nu]$. But then, again by $(\alpha_5),$ we have
$$
\epsilon \frac{d\ol Y}{dt} - g(t, \ol Y,\epsilon) = g_y(t,0,\e)\ol Y -g(t, \ol Y,\epsilon) \geq 0
$$
and $\ol Y$ is an upper solution of (\ref{eq1}). Hence
$$
0\leq \lim_{\e\to 0^+}y(t,\e)\leq \lim_{\e\to 0^+}\ol Y(t,\e) =0
$$
uniformly on $[t_0+\nu, t^*-\nu]$. Since $\nu$ was arbitrary, we obtain the first identity of (\ref{btz2}).

We can also derive an upper bound for $y(t,\e)$ for $t\in [t^*-\nu, T]$. From the above, there is $\bar \e$ such that for $\e<\bar \e$ we have $y(t^*-\nu,\e) < \phi(t^*-\nu).$ Then, as in \cite[p. 203]{tivasv} (see also \ref{siom}), we fix  (sufficiently small) $\omega$ and select $\hat\e$ so that any solution $y(t,\e)$ with $\e<\hat \e$ that enters the strip $\{(y,t);\; t\in [t^*-\nu, T], \phi(t)-\omega<y<\phi(t)+\omega\}$, stays there. Hence, we have
$$
y(t,\e)\leq \phi(t)+\omega, \quad t\in [t^*-\nu, T],
$$
for any $\e\leq \min\{\bar \e, \hat\e\}.$

To prove the second identity of (\ref{btz2}) we first have to prove that $y(t,\e)$ detaches from zero soon after $t^*$. Clearly, $\ov Y(t,\e)$ has this property as $G(t,\e)>0$ for $t>t^*$. However, this is an upper solution so its behaviour does not give any indication about the properties of $y(t,\e)$. Hence, we consider the function (\ref{ulY}), $\ul Y(t,\e) = \eta \exp \e^{-1}(G(t,\e)-\delta(t-t_0)), $
with $\eta \leq \min\{\mathring y, \min_{t\in [t^*,T]} \phi(t)\}.$ Using assumptions $(\alpha_2)$ and $(\alpha_4)$ and the implicit function theorem (first for $G(t,0)-\delta (t-t_0)$ and then for $G(t,\e)-\delta(t-t_0)$) we find that for any sufficiently small $\delta$ there exists $\e(\delta),$ such that for any $0<\e<\e(\delta)$ there is a simple root $t(\delta, \e)> t^*$  of $G(t,\e)-\delta(t-t_0) = 0$. Moreover, $t(\delta,\e)\to t^*$ as $\delta,\e \to 0$. Then we have
\begin{equation}
\ul Y(t,\e)\leq \eta \quad\mathrm{for\;} t_0\leq t\leq t(\delta,\e)
\label{2.18}
\end{equation}
with $\ul Y(t(\delta,\e),\e)=\eta.$ On the other hand
$$
\epsilon \frac{d\ul Y}{dt} - g(t, \ul Y,\epsilon) = g_y(t,0,\e)\ul Y -g(t, \ul Y,\epsilon) - \delta \ul Y.
$$
Since $0\leq \eta \leq \mathring y\leq c_0$ (see the first part of the proof), for any $y\in [0,c_0]$ we obtain, by assumption $(\alpha_3),$
$$
g(t,y,\e) = g_y(t,0,\e)y + \frac{1}{2}g_{yy}(t,y^*,\e)y^2
$$
with $0\leq y^*\leq c_0.$ Then
\begin{equation}
g_y(t,0,\e)y - g(t,y,\e) = - \frac{1}{2}g_{yy}(t,y^*,\e)y^2 \leq k y^2
\label{k}
\end{equation}
for $k = \sup_{\bar D} |g_{yy}|<\infty$ and hence
$$
\epsilon \frac{d\ul Y}{dt} - g(t, \ul Y,\epsilon) = k^2 \ul Y^2 - \delta \ul Y\leq 0
$$
on $[t_0,t(\delta,\eta)],$ provided $\eta \leq \delta/k$. Observe, that the constants are correctly defined. Indeed, $k$ depends on the properties of $g$ that are independent of $\e,$ and on $c_0,$ that is selected a priori as the constant for which assumption $(\alpha_5)$ is satisfied. Thus, it is independent of $\delta$ and $\eta$. Next, we can fix $\delta$ and $\e(\delta)$ which are related to solution of $G(t,\e)-\delta(t-t_0) = 0$ and independent of $\eta$. Finally, we can select $\eta$ to satisfy the above condition. Thus, $\ul Y$ is a subsolution of (\ref{eq1}) on $[t_0,t(\delta,\e)].$

Next we have to make these considerations independent of $\e$. Since the solution $t(\delta,\e)$ is a $C^1$ function, for a fixed $\delta$ we can consider $t(\delta) = \sup_{0<\e \leq \e(\delta)} t(\delta,\e).$ As before, $t(\delta)\to t^*$ as $\delta \to 0$. By the regularity of $g$ and second part of assumption $(\alpha_2)$ we see that $g(t,\eta,0)>0$ on $[t^*,T]$ for sufficiently small $\eta>0$ and then $g(t,\eta,\e)>0$ for sufficiently small $\e$ on $[t^*,T]$. Thus, $\ul Y(t,\e) = \eta$ is a subsolution on $[t(\delta,\e), t(\delta)]$. Hence we see that
\begin{equation}
\eta \leq y(t(\delta), \e)\leq \phi(t(\delta)) +\omega
\label{subsup}
\end{equation}
for sufficiently small  $\omega$ and for sufficiently small corresponding $\e$. Clearly, the points $(t(\delta),\eta)$ and $(t(\delta), \phi(t(\delta)) +\omega)$ are in the basin of attraction of $\phi$ and hence solutions originating from these two points converge to $\phi$ for $t>t(\delta)$. Since solutions cannot intersect we have, by (\ref{subsup}),
\begin{equation}
\lim\limits_{\e\to 0^+} y(t,\e) = \phi(t),\quad \bar t>t(\delta)
\label{limout}
\end{equation}
uniformly on $[\bar t, T]$ and thus the convergence is almost uniform on $]t(\delta), T]$. Since, however,  $t(\delta)\to t^*$ as $\delta \to 0$, we obtain  the second identity of (\ref{btz2}).

 {A closer scrutiny of the proof shows that the assumption that $g$ is a $C^2$ function with respect to all variables is too strong. Indeed, for (\ref{Glip}) we need that $g_y(t,0,\e)$ be  Lipschitz continuous in $\e \in I_{\e_0}$ uniformly in $t \in [t_0,t^*]$. Further, (\ref{k}) together with earlier calculations require $g$ to be twice continuously differentiable with respect to $y$. Finally, the construction of the root $t(\delta,\e)$ requires $G$ to be a $C^1$ function in some neighborhood of $(t^*, \e)$ for which it is sufficient that $g_u(t,0,\e)$ be a $C^1$ function in $\e$ for sufficiently small $\e,$ uniformly in $t$ in a neighbourhood of $t^*$.}


\begin{thebibliography}{99.}
\bibitem{Am} Amman, H.; Escher, J. Analysis II. Birkh\"{a}user, Basel (2008).
\bibitem{BaLa14} Banasiak, J.; Lachowicz, M. Methods of Small Parameter in Mathematical Biology. Birkh\"{a}user/Springer, Heidelberg/New York (2014).
\bibitem{BaK1} Banasiak, J.; Kimba Phongi, E.; Lachowicz, M. A singularly perturbed SIS model with age structure, Mathematical Biosciences and Engineering, \textbf{10}(3), (2013), 499-521.
\bibitem{BaK2} Banasiak, J.; Kimba Phongi, E. Canard-type solutions in epidemiological models,  Discrete Contin. Dyn. Syst. 2015, Dynamical systems, differential equations and applications, 10th AIMS Conference, Suppl., 85--93.
    \bibitem{BeCa} Beno\^it, E.;  Callot, J.-L.;  Diener, F.;  Diener, M. Chasse au canard, Collect. Math. \textbf{32}, (1981), 37--119.
\bibitem{BoTe1}  Boudjellaba, H.; Sari, T. Stability Loss Delay in Harvesting Competing
Populations, J. Differential Equations, \textbf{152}, (1999), 394--408.
\bibitem{BoTe2}  Boudjellaba, H.; Sari, T. Dynamic transcritical bifurcations in a class of slow--fast
predator--prey models, J. Differential Equations, \textbf{246}, (2009), 2205--2225.
\bibitem{but2004}  Butuzov, V. F.; Nefedov, N. N.; Schneider, K. R. Singularly perturbed problems in case of exchange of stabilities. Differential equations. Singular perturbations. J. Math. Sci. (N. Y.) \textbf{121} (2004), no. 1, 1973--2079.
    \bibitem{Braun} Braun, M. Differential Equations and Their Applications, 3rd ed., Springer, New York, 1983.
    \bibitem {E1} Eckhaus, W. Relaxation oscillations including a standard chase on french ducks,  in:  F. Verhulst (ed.), Asymptotic analysis II  LNM 985, Springer, New York, (1983), 449-"1¤74.
        \bibitem{Fe} Fenichel, F. Geometric Singular Perturbation Theory for Ordinary Differential Equations,
J. Differential Equations, \textbf{31}(1), 53--98, (1979).
\bibitem{Hab} Haberman, R. Slowly varying jump and transition phenomena associated with algebraic bifurcation problems, SIAM J. App. Math. \textbf{37}, (1979), 69-106.
        \bibitem{HW}  Hairer, E.; Wanner, G. Solving Ordinary Differential Equations II. Springer, Berlin (1991).
        \bibitem{Hek} Hek, G. Geometrical singular perturbation theory in biological practice, J. Math. Biol.,
\textbf{60}, (2010), 347--386.
    \bibitem{Ho} Hoppenstead, F. Stability in Systems with Parameter, J. Math. Anal. Appl.\textbf{18}, (1967), 129--134.
   \bibitem{Jo}  Jones, C.K.R.T. Geometric singular perturbation theory, in: L. Arnold (Ed.), Dynamical Systems, Lecture Notes in Math.,
vol. 1609, Springer-Verlag, Berlin, (1994), pp. 44--118.
        \bibitem{EKP} Kimba Phongi, E. On Singularly Perturbed Problems and Exchange of Stabilities. PhD Thesis, UKZN, (2015).
\bibitem{KS} Krupa, M.; Szmolyan, P. Extending geometric singular perturbation theory to non--hyperbolic points--fold and canard points in two dimensions, SIAM J. Math. Anal., \textbf{33}(2), (2001), 286--314.
    \bibitem{Kue} Kuehn, Ch. Multiple Time Scale Dynamics, Springer, Cham, (2015).
        \bibitem{LS1} Lebovitz, N. R.;  Schaar, R.J. Exchange of stabilities in autonomous systems, Stud. Appl. Math.,
\textbf{54},  (1975), 229-"1¤70.
\bibitem{LS2} Lebovitz, N. R.;  Schaar, R.J. Exchange of stabilities in autonomous systems, II. Vertical bifurcation,
Stud. Appl. Math., \textbf{56},  (1977), 1--50.
\bibitem{Mis} Mishchenko, E.F.; Rozov, N. Kh. Differential Equations with Small Parameter and Relaxation Oscillations, Plenum Press, New York, (1980).
\bibitem{Nei}  Neishtadt, A. I. On delayed stability loss under dynamic bifurcations, I, Differ. Uravn., \textbf{23},  (1987), 2060--
2067.
        \bibitem{OM}  O'Malley, Jr., R. E. Singular Perturbation Methods for Ordinary
Differential Equations, Springer, New York, (1991).
\bibitem{Rin1} Rinaldi, S.; Muratori, S.  Slow–fast limit cycles in predator–prey models. Ecol Model. \textbf{61}, (1992), 287--308.
\bibitem{Sa} Sakamoto, K. Invariant manifolds in singular perturbation problems for ordinary differential equations, Proc. Roy. Soc. Edin. \textbf{116A}, (1990), 45--78.
    \bibitem{Sch} Schecter, S. Persistent unstable equilibria and closed orbits of singularly perturbed equations, J. Differential Equations, \textbf{60}, (1985), 893--898.
    \bibitem{Walt} Smith H.L.; Waltman, P. The theory of chemostat, Cambridge University Press, Cambridge, (1995).
    \bibitem{Scch} Shchepkina, E.; Sobolev, V.; Mortell, M.P. Singular Perturbations. Introduction to System Order Reduction Methods with Applications,  LNM 2114, Springer, Cham, (2014).
    \bibitem{Shi} Shishkova, M. A. Study of a system of differential equations with a small parameter at the highest derivatives, Dokl. Akad.
Nauk SSSR 209 (1973) 576"1¤79.
\bibitem{tivasv} Tikhonov, A.N.; Vasilyeva, A.B.; Sveshnikov, A.G. Differential Equations.
Nauka, Moscow (1985), in Russian; English translation: Springer Verlag, Berlin (1985).
\end{thebibliography}
\end{document}